\documentclass[reqno,a4paper]{amsart}
\usepackage[utf8]{inputenc}

\usepackage{pgf,tikz,pgfplots}
\usetikzlibrary{arrows}

\usepackage{amsmath,amssymb,mathtools}
\usepackage{mathrsfs}
\usepackage{yhmath}

\usepackage{booktabs}

\usepackage[font=footnotesize,labelfont=bf]{caption}

\newtheorem{theorem}{Theorem}
\newtheorem{proposition}{Proposition}
\newtheorem{corollary}{Corollary}
\newtheorem{lemma}{Lemma}

\newcommand{\de}[0]{\mathrel{\mathop:}=}
\newcommand{\ie}[0]{\mathrm{i}}
\newcommand{\dif}[1]{\mathrm{d}#1}
\newcommand{\R}[0]{\mathbb{R}}
\newcommand{\C}[0]{\mathbb{C}}
\newcommand{\Z}[0]{\mathbb{Z}}
\newcommand{\N}[0]{\mathbb{N}}
\newcommand{\fl}[1]{\left\lfloor{#1}\right\rfloor}
\newcommand{\sgn}[1]{\mathrm{sgn}\left(#1\right)}

\title[Explicit zero density estimate]{Explicit zero density estimate for the Riemann zeta-function near the critical line}
\author{Aleksander Simoni\v{c}}
\address{School of Science, The University of New South Wales (Canberra), ACT, Australia}
\email{a.simonic@student.adfa.edu.au}
\date{\today}

\begin{document}

\begin{abstract}
    In 1946, A.~Selberg proved $N(\sigma,T)\ll T^{1-\frac{1}{4}\left(\sigma-\frac{1}{2}\right)}\log{T}$ where $N(\sigma,T)$ is the number of nontrivial zeros $\rho$ of the Riemann zeta-function with $\Re\{\rho\}>\sigma$ and $0<\Im\{\rho\}\leq T$. We provide an explicit version of this estimate, together with an explicit approximate functional equation and an explicit upper bound for the second power moment of the zeta-function on the critical line.
\end{abstract}

\maketitle
\thispagestyle{empty}

\section{Introduction}

Let $\zeta(s)$ be the Riemann zeta-function and denote by $\rho=\beta+\ie\gamma$ a nontrivial zero of $\zeta(s)$ in the critical strip $0\leq\Re\{s\}\leq1$. Denote by $N(T)$ the number of zeros $\rho$ with $\gamma\in(0,T]$, and let $N(\sigma,T)$ be the number of those zeros with $\beta>\sigma\geq1/2$. Trivially, $N(\sigma,T)\leq \frac{1}{2}N(T)$, where 
\begin{equation}
\label{eq:RvM}
\left|N(T)-\frac{T}{2\pi}\log{\frac{T}{2\pi e}}-\frac{7}{8}\right| \leq 0.11\log{T}+0.29\log{\log{T}}+2.29+\frac{1}{5T},
\end{equation}
$T\geq e$, is an explicit version of the Riemann--von Mangoldt formula, see \cite[Corollary 1]{PTZeta} and \cite[Corollary 1]{TrudgianUpper}. The Riemann Hypothesis is equivalent to $N(1/2,T)=0$ for every $T>0$. It has been rigorously verified for all nontrivial zeros with $\left|\Im\{\rho\}\right|\leq H_0$, where
\[
H_0 \de 3.0610046\cdot10^{10},
\]
the result due to Platt\footnote{Platt and Trudgian will soon announce that $H_0$ can be replaced by $2.5\cdot10^{12}$, see \cite[Lemma 4]{PTErrorTerm}.}, see \cite{PlattRH}. Non-trivial upper bounds for $N(\sigma,T)$ are called \emph{zero density estimates}. There exist many such estimates in the literature, for instance Ingham's theorem 
\begin{equation}
\label{eq:Ingham}
N(\sigma,T)\ll T^{\frac{3(1-\sigma)}{2-\sigma}}\log^5{T}.
\end{equation}
There are other zero density estimates which are better than \eqref{eq:Ingham} in smaller regions of the critical strip. Possible applications strongly depend on the position of such a region, e.g., to the distribution of prime numbers if $\sigma$ is close to $1$, see \cite{PTErrorTerm}, and to problems connected to the function $S(t)$ and to the pair correlation conjecture when $\sigma$ is close to $1/2$. We refer the reader to \cite{KLN} and references therein for zero density estimates near the one-line. Much less work was done in the latter case. Selberg proved in \cite[Theorem 1]{SelbergContrib} that 
\begin{equation}
\label{eq:selberg}
N(\sigma,T)\ll T^{1-\frac{1}{4}\left(\sigma-\frac{1}{2}\right)}\log{T},
\end{equation}
which supersedes \eqref{eq:Ingham} for $\sigma-1/2\ll \log{\log{T}}/\log{T}$. In fact, he provided a bound for $N(\sigma,T+H)-N(\sigma,T)$ where $H\in[T^a,T]$ and $a\in(1/2,1]$, such that \eqref{eq:selberg} is a special case of $H=T$ after a dyadic partition. Later Jutila improved in \cite{Jutila} the constant $1/4$ to $1-\varepsilon$ and Conrey announced in \cite{ConreyAtLeast} a further improvement to $8/7-\varepsilon$, where the implied constant in \eqref{eq:selberg} depends on $\varepsilon>0$. Observe that if one could prove \eqref{eq:selberg} with $2$ instead of $1/4$, this would imply the Density Conjecture.

There exist only a few explicit zero density estimates, e.g., in recent papers \cite{KadiriZeroDensity} and \cite{KLN} where they improve older results by Cheng and Ramar\'{e}, and are good near the one-line. As an example we mention only
\begin{equation}
\label{eq:KLN}
N\left(\sigma,T\right)\leq A(\sigma)\cdot T^{\frac{8}{3}(1-\sigma)}\log^{5-2\sigma}{T}+B(\sigma)\log^2{T}, 
\end{equation}
valid for $\sigma\in[3/5,1)$ and $T\geq H_0$, see \cite{KLN}, where $A(\sigma)$ and $B(\sigma)$ are positive and calculable\footnote{Note that the first column in Table 1 in \cite{KLN} should have $\sigma$ in place of $\sigma_0$. It seems that $A(\sigma)$ and $B(\sigma)$ are increasing and decreasing functions, respectively. The author thanks Allysa Lumley for calculating $A(0.638)=2.789\ldots$ and $B(0.638)=5.312\ldots$.} functions, e.g., $A(37/58)\leq2.9$ and $B(37/58)\leq5.6$. However, \eqref{eq:KLN} produces non-trivial bound for $\sigma>5/8$. It seems that the only explicit result of Selberg-type zero density estimate was done by Karatsuba and Korol\"{e}v in \cite[Theorem 1]{KaratKor}. They proved
\[
N(\sigma,T+H)-N(\sigma,T-H)\leq 13HT^{\varepsilon(1-2\sigma)/10}\log{T}
\]
for $0<\varepsilon<0.001$, $T\geq T_0(\varepsilon)>0$ and $H=T^{27/82+\varepsilon}$. Unfortunately, $T_0(\varepsilon)$ is not explicitly known.

The main result of this paper is the following explicit version of \eqref{eq:selberg}.

\begin{theorem}
\label{thm:main}
Let $T\geq H_0$ and $\sigma\in[1/2,0.831]$. Then we have
\[
N(\sigma,2T)-N(\sigma,T) \leq a T^{1-\frac{1}{4}\left(\sigma-\frac{1}{2}\right)}\log{T}+b\log^2{T}+c\log{T}\log{\log{T}}+d\log{T},
\]
with $a=10395.2$, $b=1.104$, $c=0.173$ and $d=0.51$.
\end{theorem}

Under the assumptions of Theorem \ref{thm:main}, an immediate corollary is 
\begin{equation*}
N(\sigma,T) \leq \frac{10395.21}{2^{1-\frac{1}{4}\left(\sigma-\frac{1}{2}\right)}-1}T^{1-\frac{1}{4}\left(\sigma-\frac{1}{2}\right)}\log{\frac{T}{2}}
\end{equation*}
for $T\geq 2H_0$. In virtue of \eqref{eq:KLN}, this bound is of interest only for $\sigma\in(1/2,5/8]$. Nevertheless, for $\sigma=37/58$, when exponents of $T$ in both inequalities are equal, it is better than \eqref{eq:KLN}. For larger $H_0$ we can obtain smaller values for the leading constant, e.g., for $T\geq10^{50}$ and $\sigma\in[1/2,0.569]$ we have
\begin{equation}
\label{eq:secondbound}
N(\sigma,2T)-N(\sigma,T) \leq 5.357\cdot T^{1-\frac{1}{4}\left(\sigma-\frac{1}{2}\right)}\log{T} + 1.11\cdot \log^2{T}.
\end{equation}
But with the method presented here we cannot get a smaller constant than $3.259$.

Our approach to Theorem \ref{thm:main} strongly relies on Selberg's original proof with the simplification $H=T$. The main idea is using the approximate functional equation (Theorem \ref{thm:afe}) to prove the second power moment of $\zeta(s)$ with a special weight (Theorem \ref{thm:SelbergMoment}), which is then used to estimate the main term in Littlewood's zero-counting lemma for Selberg's mollifier (Proposition \ref{prop:Phi}). These three crucial steps constitute Sections \ref{sec:eafe}, \ref{sec:espm} and \ref{sec:eszd}, respectively. Beside the proof of Theorem \ref{thm:main}, which is presented in Section \ref{sec:proof}, we also provide three additional results which might be interesting on their own, namely explicit versions of the approximate functional equations for $\zeta(s)$ and $\zeta^2(s)$, see Theorem \ref{thm:explmain} and Corollaries \ref{cor:explmain} and \ref{cor:eafesq}, and an explicit upper bound for the second power moment of $\zeta(s)$ on the critical line
\[
\int_{0}^{T}\left|\zeta\left(\frac{1}{2}+\ie t\right)\right|^2\dif{t} \leq T\log{T}-\left(1+\log{2\pi}-2\gamma\right)T+70.26\cdot T^{\frac{3}{4}}\sqrt{\log{\frac{T}{2\pi}}},
\]
valid for $T\geq2000$, see Corollary \ref{cor:espm} for a more precise statement. All results are believed to be new, and the latter inequality greatly improves the recently announced estimate \cite[Theorem 4.3]{DHA}. 

\section{Explicit approximate functional equation}
\label{sec:eafe}

We can approximate $\zeta(s)$ with Dirichlet polynomials to arbitrary precision on every compact set in $\Re\{s\}>1$. Hardy and Littlewood showed in \cite[Lemma 2]{HLcritical} that this is also possible to some extent in the critical strip. 

\begin{theorem}
\label{thm:fafe}
Let $s=\sigma+\ie t$ where $\sigma\in(0,1]$ and $s\neq 1$. Also assume that $x\geq1$ and $|t|<2\pi x$. Then
\begin{equation}
\label{eq:fafeg}
\zeta(s)=\sum_{n\leq x}n^{-s} - \frac{x^{1-s}}{1-s} + R(s;x),
\end{equation}
where $R(s;x)=O\left(x^{-\sigma}\right)$ uniformly.
\end{theorem}

In many cases the sum in \eqref{eq:fafeg} has too many terms to be useful. Remember that the functional equation for the Riemann zeta-function is $\zeta(s)=\chi(s)\zeta(1-s)$ where
\begin{equation}
\label{eq:chi}
\chi(s)=2^s\pi^{s-1}\sin\left(\frac{\pi s}{2}\right)\Gamma{(1-s)}.
\end{equation}
Hardy and Littlewood proved in \cite[Theorem A]{HLapprox1} the following refinement of \eqref{eq:fafeg} which is known as the \emph{approximate functional equation}.

\begin{theorem}
\label{thm:afe}
Let $s=\sigma+\ie t$ where $\sigma\in[0,1]$ and $|t|\geq 2\pi$. Also assume that $2\pi xy=|t|$ for $x,y\geq 1$. Then
\begin{equation}
\label{eq:afe}
\zeta(s) = \sum_{n\leq x}n^{-s} + \chi(s) \sum_{n\leq y}n^{s-1} + R_1\left(s;x,y\right), 
\end{equation}
where $R_1\left(s;x,y\right)=O\left(x^{-\sigma}+y^{\sigma-1}|t|^{1/2-\sigma}\right)$ uniformly.
\end{theorem}

Equation \eqref{eq:afe} first appeared in \cite{HLcritical}, but with a factor $\log{|t|}$ in the remainder. The proof of this ``imperfect'' approximate functional equation exploits the Poisson summation formula\footnote{See also \cite[pp.~79--80]{Titchmarsh} and \cite[Chapter III]{KaratsubaVoronin}.} while their approach to Theorem~\ref{thm:afe} was complex analytic in the sense that they used contour integration; it is sketched in \cite[p.~81]{Titchmarsh} where also Theorem \ref{thm:fafe} is proved in such a way. Later they provided in \cite{HLapprox2} a proof along the similar lines as in \cite{HLcritical}. However, the more common proof, see \cite[pp.~82--84]{Titchmarsh} or \cite[pp.~99--104]{Ivic}, has roots in the celebrated paper of Siegel \cite{Siegel} where he developed Riemann's ideas on the zeta function and derived 
\begin{equation}
\label{eq:IRS}
\zeta(s)=\mathcal{R}(s) + \chi(s)\overline{\mathcal{R}(1-\bar{s})},
\end{equation}
where $\mathcal{R}(s)$ is some function given as the contour integral. A more useful expression for this function is
\begin{equation}
\label{eq:RS}
\mathcal{R}(s) = \sum_{n\leq\sqrt{\frac{|t|}{2\pi}}}n^{-s}+\left(\frac{|t|}{2\pi}\right)^{-\frac{\sigma}{2}}E_{\mathrm{L}}(s),
\end{equation}
where $E_{\mathrm{L}}(s)$ has a known asymptotic expansion in powers of $|t|^{-1/2}$, see \cite[Theorem 3.1]{deReyna}. Equation \eqref{eq:RS}, now called the \emph{Riemann--Siegel formula}, was first proposed by Lehmer in \cite{LehmerExt} for values on the critical line. Equations \eqref{eq:IRS} and \eqref{eq:RS} imply \eqref{eq:afe} in the symmetric case $x=y=\sqrt{|t|/(2\pi)}$. 

The Riemann--Siegel formula can be used to calculate values of $\zeta(s)$ relatively fast, e.g., through the Odlyzko--Sch\"{o}nhage algorithm which is suitable for large scale computations, and thus it replaced the previous method based on the Euler--Maclaurin summation formula or on its simpler version \eqref{eq:fafeg}. For high precision calculations we still need to know explicit bounds. Titchmarsh \cite{TitchTheZeros} carried out a complete analysis of the error terms which comes from Siegel's method. Since his estimates are most suitable only for sufficiently large values of $t$, Turing developed a different method, see \cite{Turing}. Gabcke provided in \cite{Gabcke} good bounds for $E_{\mathrm{L}}(s)$ in case of $\sigma=1/2$ and Arias de Reyna \cite[Section 4]{deReyna} for all values in the critical strip. There also exist generalisations of \eqref{eq:IRS} to $L$-functions and to specially designed smooth functions, see \cite{Hiary}. In the context of the Knopp--Hasse--Sondow formula for $\zeta(s)$, it is possible to obtain even better error term, see \cite{Jerby}.

None of the previously mentioned authors considered explicit versions of equation \eqref{eq:afe} in the non-symmetrical case. The main result of this section is an explicit form of the approximate functional equation (Theorem \ref{thm:eafe}) which comes from the standard proof of Theorem \ref{thm:afe}. The main advantage of this is having a uniform bound on constants in the $O$-estimate of $R_1$, independent of $x$ and $y$. In Section \ref{sec:numerics} we prove the following. 

\begin{theorem}
\label{thm:explmain}
Let $s=\sigma+\ie t$ where $\sigma\in[1/2,1]$ and $|t|\geq 2\pi$. Also assume that $2\pi xy=|t|$ for $x,y\geq 1$. If $R_1(s;x,y)$ is defined by equation \eqref{eq:afe}, then 
\[
\left|R_1(s;x,y)\right| \leq E\cdot x^{-\sigma} + F\cdot\left(\frac{|t|}{2\pi}\right)^{\frac{1}{2}-\sigma}y^{\sigma-1},
\]
where $E$ and $F$ are non-negative real numbers, whose values are given by Table \ref{tab:ub} for $|t|\geq 2\pi$, Table \ref{tab:large} for $|t|\geq 10^3$ and Table \ref{tab:verylarge} for $|t|\geq 10^{10}$.
\end{theorem}

We used bounds from \cite{deReyna} to give constants in Tables \ref{tab:ub} and \ref{tab:large} in the symmetric case. While these are expected to be better than those obtained by the classical method, they are not so large at all.

\begin{table}[h!]
\begin{minipage}[b]{.45\textwidth}
   \centering
\begin{tabular}{cccc}
\toprule
& $x\leq y$ & $x>y$ & $x=y$ \\ 
\midrule
$E$ & $36.094$ & $0$ & $4.257$  \\
$F$ & $0$ & $127.126$ & $0$  \\ 
\midrule
$\widetilde{E}$ & $36.214$ & $0$ & $4.376$  \\
$\widetilde{F}$ & $0$ & $127.245$ & $0$  \\ 
\bottomrule
\end{tabular}
   \caption{Bounds for $|t|\geq 2\pi$.}
   \label{tab:ub}
\end{minipage}
\hfill
\begin{minipage}[b]{.45\textwidth}
   \centering
\begin{tabular}{cccc}
\toprule
& $x\leq y$ & $x>y$ & $x=y$ \\ \midrule
$E$ & $10.983$ & $0$ & $1.195$ \\
$F$ & $0$ & $15.726$ & $0$ \\ 
\midrule
$\widetilde{E}$ & $10.992$ & $0$ & $1.205$  \\
$\widetilde{F}$ & $0$ & $15.726$ & $0$  \\ 
\bottomrule
\end{tabular}
   \caption{Bounds for $|t|\geq 10^3$.}
   \label{tab:large}
\end{minipage}
\end{table}

\begin{table}[h!]
\begin{minipage}[b]{.45\textwidth}
   \centering
\begin{tabular}{cccc}
\toprule
& $x\leq y$ & $x>y$ & $x=y$ \\ 
\midrule
$E$, $\widetilde{E}$ & $10.7502$ & $0$ & $1.00007$  \\
$F$, $\widetilde{F}$ & $0$ & $15.203$ & $0$  \\ 
\bottomrule
\end{tabular}
   \caption{Bounds for $|t|\geq 10^{10}$.}
   \label{tab:verylarge}
\end{minipage}
\end{table}

Sometimes it is more convenient to have \eqref{eq:afe} in the form
\begin{equation}
\label{eq:chitilde}
\zeta(s) = \sum_{n\leq x}n^{-s} + \widetilde{\chi}(s) \sum_{n\leq y}n^{s-1} + \widetilde{R}_1(s;x,y)
\end{equation}
where 
\begin{equation}
\label{eq:tildechi}
\widetilde{\chi}(\sigma+\ie t)\de \left(\frac{2\pi}{|t|}\right)^{\sigma-\frac{1}{2}}\left(\frac{|t|}{2\pi e}\right)^{-\ie t}e^{\sgn{t}\frac{\pi}{4}\ie}.
\end{equation}
Note that a consequence of Stirling's formula is $\chi(\sigma+\ie t)\sim\widetilde{\chi}(\sigma+\ie t)$ for $t\to\infty$ where $\chi(s)$ is defined by \eqref{eq:chi}. In Section \ref{sec:stirling} we will provide an explicit version of this asymptotic relation, see Proposition \ref{prop:chi}. This will enable us to prove the following corollary of Theorem \ref{thm:explmain}.

\begin{corollary}
\label{cor:explmain}
Let $s=\sigma+\ie t$ where $\sigma\in[1/2,1]$ and $|t|\geq 2\pi$. Also assume that $2\pi xy=|t|$ for $x,y\geq 1$. If $\widetilde{R}_1(s;x,y)$ is defined by equation \eqref{eq:chitilde}, then 
\begin{equation}
\label{eq:explmaintild}
\left|\widetilde{R}_1(s;x,y)\right| \leq \widetilde{E}\cdot x^{-\sigma} + \widetilde{F}\cdot\left(\frac{|t|}{2\pi}\right)^{\frac{1}{2}-\sigma}y^{\sigma-1},
\end{equation}
where $\widetilde{E}$ and $\widetilde{F}$ are non-negative real numbers, whose values are given by Table \ref{tab:ub} for $|t|\geq 2\pi$, Table \ref{tab:large} for $|t|\geq 10^3$ and Table \ref{tab:verylarge} for $|t|\geq 10^{10}$.
\end{corollary}

\subsection{Some estimates for $R(s;x)$.} 

It seems that an explicit version of Theorem \ref{thm:fafe} first appeared in \cite[Proposition 1]{Cheng}. 
Cheng's result was considerably improved by Kadiri in \cite[Theorem 1.2]{KadiriZeroDensity}. Following the proof outlined there, we can obtain an explicit bound for $R(s;x)$ which also slightly improves Kadiri's bound. 

\begin{theorem}
\label{thm:efafe}
With assumptions and notations as in Theorem \ref{thm:fafe} we have
\begin{equation}
\label{eq:fafe1}
\left|R(s;x)\right| \leq x^{-\sigma}\left(\frac{1}{2}+\frac{3x}{|t|}\sqrt{1+\left(\frac{\sigma}{t}\right)^2}\left(1-\frac{t}{2x}\cot{\frac{t}{2x}}\right)\right)
\end{equation}
for $t\neq0$.
\end{theorem}

\begin{proof}
Our proof is basically the same as the proof in \cite{KadiriZeroDensity}, except that we use closed expression for the sum in \eqref{eq:est2}. 

Let $N\geq 2$. We start with the classical summation formula
\[
\sum_{x<n\leq N} n^{-s} = \frac{N^{1-s}}{1-s}-\frac{x^{1-s}}{1-s}-\frac{((x))}{x^s}+\frac{1}{2N^s}+s\int_x^N \frac{((u))}{u^{s+1}}\dif{u}
\]
where $((x))\de \lfloor x\rfloor-x+1/2$, see \cite[Equation 2.1.2]{Titchmarsh}. Then 
\[
R(s;x) = - \frac{((x))}{x^s} + s\int_N^\infty \frac{((u))}{u^{s+1}}\dif{u} + s\int_x^N \frac{((u))}{u^{s+1}}\dif{u}
\]
and from this it follows that
\begin{equation}
\label{eq:est1fafe}
|R(s;x)| \leq \frac{|s|}{2\sigma N^{\sigma}} + \frac{1}{2x^\sigma} + |t|\sqrt{1+\left(\frac{\sigma}{t}\right)^2}\left|\int_x^N \frac{((u))}{u^{s+1}}\dif{u}\right|.
\end{equation}
Writing $((u))$ in form of the Fourier series and applying the second mean value theorem, we have
\[
\int_x^N \frac{((u))}{u^{s+1}}\dif{u} = \sum_{n=1}^\infty \frac{I(n)-I(-n)}{n}, \quad |I(\pm n)|\leq \frac{3}{2\pi} \frac{x^{-\sigma}}{2\pi xn\mp t}.
\]
For details of this derivation see \cite[pp.~189--190]{KadiriZeroDensity}. Then
\begin{equation}
\label{eq:est2}
\left|\int_x^N \frac{((u))}{u^{s+1}}\dif{u}\right| \leq \frac{6x^{1-\sigma}}{(2\pi x)^2}\sum_{n=1}^\infty \left(n^2-\left(\frac{t}{2\pi x}\right)^2\right)^{-1} = \frac{3x^{1-\sigma}}{t^2}\left(1-\frac{t}{2x}\cot{\frac{t}{2x}}\right)
\end{equation}
and \eqref{eq:fafe1} clearly follows from \eqref{eq:est1fafe} and \eqref{eq:est2} after taking $N\to\infty$. Equality in \eqref{eq:est2} is established by a well-known identity 
\[
\sum_{n=1}^\infty \frac{1}{n^2-a^2} = \frac{1-(\pi a)\cot{(\pi a)}}{2a^2},
\]
see \cite[Eq.~\textbf{1.421} 3]{GradRyz}.
\end{proof}

\begin{corollary}
Let $s=\sigma+\ie t$, $\sigma\in[1/2,1]$, $|t|\geq t_0>0$ and $c>1/(2\pi)$. Then
\[
\left|\zeta(s)-\sum_{n<c|t|}n^{-s}\right| \leq (tc)^{-\sigma}\left(c+\frac{1}{2}+\frac{3c}{t_0}\sqrt{1+t_0^2}\left(1-\frac{1}{2c}\cot{\frac{1}{2c}}\right)\right).
\]
In particular, if $c=1$ and $t_0=\gamma_1$ where $\gamma_1\approx14.1347$ is the imaginary part of the first non-trivial zero of $\zeta(s)$, then
\begin{equation*}
\left|\zeta(s)-\sum_{n<|t|}n^{-s}\right| \leq 1.755\cdot t^{-\sigma}.
\end{equation*}
\end{corollary}

This improves Kadiri's constant $2.1946$, see \cite[Corollary 1.3]{KadiriZeroDensity}. It was shown in \cite[Lemma 2.10]{DHA} that the Euler--Maclaurin summation formula implies that $\eqref{eq:fafeg}$ is true with $\left|R(s;x)\right|\leq 5/6\cdot x^{-\sigma}$ for $\sigma\in(0,1]$ and $|t|\leq x$. Numerical calculations show that for $|t|\geq1.18$ inequality \eqref{eq:fafe1}  always provides the better bound.

Taking $t\to0$ in \eqref{eq:fafe1}, we obtain $|R(\sigma;x)|\leq x^{-\sigma}\left(1/2+\sigma/(4x)\right)$. However, it is possible to prove in quite elementary way that $|R(\sigma;x)|\leq x^{-\sigma}/2$ for all $\sigma\in(0,\infty)\setminus\{1\}$, see \cite[Lemma 2.9]{DHA}. We will use this estimate in the proof of Theorem \ref{thm:SelbergMoment}, see Section \ref{sec:bounds12}.

\subsection{Explicit Stirling approximation of $\chi(s)$.}
\label{sec:stirling}

In the proof we make use of the following upper and lower bounds of $\arctan{x}$ which are asymptotically sharp. We note that the second inequality in \eqref{eq:arctan} can be found in \cite[Corollary V.14]{Alirezaei}.

\begin{lemma}
For $x\geq0$ we have
\begin{equation}
\label{eq:arctan}
\frac{\frac{\pi}{2}x}{\frac{2}{\pi}+x}\geq\arctan{x}\geq \frac{\frac{\pi}{2}x}{\frac{2}{\pi}+\sqrt{x^2+\left(\frac{\pi}{2}-\frac{2}{\pi}\right)^2}}.
\end{equation}
\end{lemma}

\begin{proof}
Denote by $\Delta_1(x)$ the difference between the upper bound and $\arctan{x}$, and by $\Delta_2(x)$ the difference between $\arctan{x}$ and the lower bound. For $x\geq0$ these functions are smooth, and we have $\Delta_1(0)=\Delta_2(0)=0$ and $\lim_{x\to\infty}\Delta_1(x)=\lim_{x\to\infty}\Delta_2(x)=0$. Numerical verification reveals that both functions are positive for $x=1$. Equations $\Delta_1'(x)=0$ and $\Delta_2'(x)=0$ can be reduced to a linear and a quadratic equation, respectively. After simple calculations we can conclude that both functions have only one stationary point on the interval $(0,\infty)$. Hence they cannot have any zeros for $x>0$ due to zero limits at infinity. This implies that both functions are positive throughout this region. 
\end{proof}

\begin{proposition}
\label{prop:chi}
Let $\sigma\in(1/2,1]$ and $|t|\geq t_0\geq1/\pi$. Then
\begin{equation*}
\chi(\sigma+\ie t) = \widetilde{\chi}(\sigma+\ie t)\left(1+\frac{C\left(\sigma,t,t_0\right)}{|t|}\right)
\end{equation*}
where
\begin{equation}
\label{eq:constant}
\left|C\left(\sigma,t,t_0\right)\right|\leq  C_1(\sigma,t)\left(1+\frac{t_0e^{-\pi t_0}}{|t|}\right)C_2\left(t\right) + C_3\left(t,t_0\right)
\end{equation}
with
\begin{flalign}
C_1\left(\sigma,t\right) &\de (1-\sigma)^2\left(\frac{1}{2}+\frac{2}{\pi}\right)+(1-\sigma)\left(\sigma-\frac{1}{2}\right)\left(\left(\frac{\pi}{2}\right)^2+\frac{1-\sigma}{2|t|}\right), \nonumber \\
C_2\left(t\right) &\de \exp{\left(\frac{1}{12|t|}+\frac{1}{90|t|^3}\right)}, \label{eq:C2} \\
C_3\left(t,t_0\right) &\de \frac{C_2\left(t_0\right)-1}{\log{C_2\left(t_0\right)}}\left(\frac{1}{12}+\frac{1}{90t^2}\right)+t_0e^{-\pi t_0}C_2(t). \nonumber 
\end{flalign}
\end{proposition}

\begin{proof}
It is enough to prove the case when $t$ is positive since $\overline{\chi(s)}=\chi\left(\bar{s}\right)$ and $\overline{\widetilde{\chi}(s)}=\widetilde{\chi}\left(\bar{s}\right)$. We use Stieltjes' explicit version of the Stirling formula for $\Gamma(z)$ where $\Re\{z\}>0$, see \cite{Olver}:
\begin{equation}
\label{eq:stieltjes}
\Gamma(z)=\sqrt{2\pi}z^{z-\frac{1}{2}}e^{-z+R(z)}, \quad |R(z)|\leq \frac{1}{12|z|} + \frac{1}{90|z|^3}.
\end{equation}
From the explicit expressions for $\chi(s)$ and $\Gamma(z)$ we obtain
\[
\frac{\chi}{\widetilde{\chi}}(\sigma+\ie t) - 1 = \left(\left(a(\sigma,t)-1\right)e^{\ie\varphi(\sigma,t)}+e^{\ie\varphi(\sigma,t)}-1\right)\left(1+\varepsilon(\sigma,t)\right)+\varepsilon(\sigma,t)
\]
where
\begin{flalign}
a(\sigma,t) &\de \left(\frac{1}{1+\left(\frac{1-\sigma}{t}\right)^2}\right)^{\frac{1}{2}\left(\sigma-\frac{1}{2}\right)}e^{r(\sigma,t)}, \nonumber \\ 
r(\sigma,t) &\de \frac{\pi}{2}t-t\arctan{\frac{t}{1-\sigma}}+\sigma-1, \label{eq:r} \\
\varphi(\sigma,t)&\de\left(\frac{1}{2}-\sigma\right)\left(\frac{\pi}{2}-\arctan{\frac{t}{1-\sigma}}\right)-\frac{t}{2}\log{\left(1+\left(\frac{1-\sigma}{t}\right)^2\right)}, \nonumber \\ 
\varepsilon(\sigma,t) &\de e^{R(1-\sigma-\ie t)}-1-e^{-\pi t+\pi\sigma \ie+R(1-\sigma-\ie t)}. \nonumber
\end{flalign}
Then
\begin{equation}
\label{eq:masterBound}
\left|\frac{\chi}{\widetilde{\chi}}(\sigma+\ie t) - 1\right| \leq \left(\left|a(\sigma,t)-1\right|+\left|\varphi(\sigma,t)\right|\right)\cdot\left|1+\varepsilon(\sigma,t)\right|+\left|\varepsilon(\sigma,t)\right|.
\end{equation}
Using Stieltjes' error term \eqref{eq:stieltjes} and noting that $\left|e^{z}-1\right|\leq e^{|z|}-1$, and that $\left(e^x-1\right)x^{-1}$ and $te^{-\pi t}$ are strictly decreasing functions for $x>0$ and $t\geq t_0\geq 1/\pi$ respectively, we get
\begin{equation}
\label{eq:epsilon}
\left|1+\varepsilon(\sigma,t)\right|\leq \left(1+\frac{t_0e^{-\pi t_0}}{t}\right)C_2(t), \quad \left|\varepsilon(\sigma,t)\right|\leq \frac{C_3\left(t,t_0\right)}{t}. 
\end{equation}


The second inequality in \eqref{eq:arctan} gives us
\begin{equation}
\label{eq:est1}
\left|\frac{\pi}{2}-\arctan{\frac{t}{1-\sigma}}\right|\leq \frac{\pi}{2}\left(1-\frac{1}{\frac{2}{\pi}\frac{1-\sigma}{t}+\sqrt{1+\left(\frac{\pi}{2}-\frac{2}{\pi}\right)^2\left(\frac{1-\sigma}{t}\right)^2}}\right)\leq\left(\frac{\pi}{2}\right)^2\frac{1-\sigma}{t}
\end{equation}
which implies
\begin{equation}
\label{eq:phi}
|\varphi(\sigma,t)|\leq\frac{1-\sigma}{t}\left(\left(\frac{\pi}{2}\right)^2\left(\sigma-\frac{1}{2}\right)+\frac{1-\sigma}{2}\right).
\end{equation}
By \eqref{eq:est1} we have
\[
\frac{\partial}{\partial t}r(\sigma,t) = \frac{\pi}{2}-\arctan{\frac{t}{1-\sigma}}-\frac{\frac{t}{1-\sigma}}{1+\left(\frac{t}{1-\sigma}\right)^2}>0.
\]
Together with $r(\sigma,0)<0$ and $\lim_{t\to\infty}r(\sigma,t)=0$ this implies $r(\sigma,t)<0$.
Next,
\begin{flalign*}
\left|a(\sigma,t)-1\right| &\leq \left(1-\left(\frac{1}{1+\left(\frac{1-\sigma}{t}\right)^2}\right)^{\frac{1}{2}\left(\sigma-\frac{1}{2}\right)}\right)e^{r(\sigma,t)}+\left|e^{r(\sigma,t)}-1\right| \\
&\leq \left(1-\exp{\left(-\frac{1}{2}\left(\sigma-\frac{1}{2}\right)\left(\frac{1-\sigma}{t}\right)^2\right)}\right)e^{r(\sigma,t)}+\left|e^{r(\sigma,t)}-1\right|.
\end{flalign*}
Applying the first inequality in \eqref{eq:arctan}, we get
\[
\left|r(\sigma,t)\right|\leq\frac{2(1-\sigma)^2}{\pi t}.
\]
Using the inequalities $e^{-x}\leq1$ and $1-e^{-x}\leq x$, both valid for $x\geq0$, we obtain
\begin{equation}
\label{eq:a}
\left|a(\sigma,t)-1\right| \leq \frac{(1-\sigma)^2}{t}\left(\frac{1}{2t}\left(\sigma-\frac{1}{2}\right)+\frac{2}{\pi}\right).
\end{equation}
Inserting \eqref{eq:a}, \eqref{eq:phi} and \eqref{eq:epsilon} into \eqref{eq:masterBound}, we finally obtain \eqref{eq:constant}.
\end{proof}

\begin{corollary}
\label{cor:factor}
Let $s=\sigma+\ie t$ where $\sigma\in(0,1)$ and $t\geq1$. Then
\[
\frac{1}{2\pi}\left|e^{-\ie\pi s}\Gamma(1-s)\right|\leq \frac{C_2(t)}{2^{\sigma}\sqrt{\pi}}t^{\frac{1}{2}-\sigma}e^{\frac{\pi}{2}t}
\]
where $C_2(t)$ is defined by \eqref{eq:C2}.
\end{corollary}

\begin{proof}
We have
\[
\left|e^{-\ie\pi s}\Gamma(1-s)\right| = \sqrt{2\pi}\left|1-\sigma-\ie t\right|^{\frac{1}{2}-\sigma}e^{\frac{\pi}{2}t+r(\sigma,t)+R(1-\sigma-\ie t)}
\]
where $r(\sigma,t)$ and $R(z)$ are defined by \eqref{eq:r} and \eqref{eq:stieltjes}, respectively. Because $t\leq\left|1-\sigma-\ie t\right|\leq 2t$, we have
$\left|1-\sigma-\ie t\right|^{\frac{1}{2}-\sigma}\leq t^{\frac{1}{2}-\sigma}$ and $\left|1-\sigma-\ie t\right|^{\frac{1}{2}-\sigma}\leq (2t)^{\frac{1}{2}-\sigma}$ for $\sigma\in[1/2,1)$ and $\sigma\in(0,1/2]$, respectively. The result now follows since $r(\sigma,t)$ is always negative.
\end{proof}

Assume that $t\geq t_0\geq 2\pi$. Observe that $C_1(\sigma,t)\leq C_1(\sigma,2\pi)$, $C_2(t)\leq C_2(2\pi)$ and $C_3(t,t_0)\leq C_3(2\pi,2\pi)$. This implies 
\begin{equation}
\label{eq:Cbound}
\left|C(\sigma,t,t_0)\right| \leq C_1(\sigma,2\pi)\left(1+e^{-2\pi^2}\right)C_2(2\pi)+C_3(2\pi,2\pi) < 0.3746
\end{equation}
since the function in the middle has the maximum at $\sigma\approx 0.54162$.

Let $\sigma\in[0,1/2)$ and $t\geq t_0\geq 2\pi$. Because $\chi(s)\chi(1-s)=1$, we have
\begin{flalign}
\label{eq:chiBoundRev}
\left|\chi(\sigma+\ie t)\right| &= \left(\frac{|t|}{2\pi}\right)^{\frac{1}{2}-\sigma}\left|1+\frac{C\left(1-\sigma,t,t_0\right)}{|t|}\right|^{-1} \nonumber \\ 
&\leq \left(\frac{|t|}{2\pi}\right)^{\frac{1}{2}-\sigma}\frac{|t|}{|t|-\left|C\left(1-\sigma,t,t_0\right)\right|}.
\end{flalign}

\subsection{Explicit estimate for $R_1(s;x,y)$.} In this section we will provide an explicit upper bound for the remainder in \eqref{eq:afe}. Our proof requires a bound of $\left|e^z-1\right|$ for $z=re^{\ie\phi}$ with $r>0$ and $\phi\in[0,2\pi]$. We would like to obtain non-zero and $\phi$-independent lower bound. Observe that the trivial estimate $\left|e^z-1\right|\geq \left|e^{r\cos{\phi}}-1\right|$ is not good since it is zero for $\phi\in\{\pi/2,3\pi/2\}$.

\begin{lemma}
\label{lem:max}
Let $\mathcal{D}_{r}\de \bigcup_{k\in\Z}\left\{z\in\C\colon |z-2k\pi\ie|<r\right\}$ where $r\in(0,\pi/\sqrt{2}]$. For $z\in\C\setminus\mathcal{D}_{r}$ we have 
\begin{equation}
\label{eq:M}
\left|e^z-1\right|\geq 1-e^{-\frac{r}{\sqrt{2}}}.
\end{equation}
\end{lemma}

\begin{proof}
Firstly, observe that $\left|e^{x+\ie y}-1\right|\geq e^x-1$ for $x>0$ and $\left|e^{x+\ie y}-1\right|\geq 1-e^x$ for $x<0$. This means that $\left|e^{x+\ie y}-1\right|\geq 1-e^{-h}$ for $|x|\geq h>0$. 

Let 
\[
\mathcal{S}_r\de\bigcup_{k\in\Z}\left\{x+\ie y\in\C\colon |x|<\frac{r}{\sqrt{2}},|y-2k\pi|<\frac{r}{\sqrt{2}}\right\}
\]
be a set of squares inscribed in $\mathcal{D}_r$. For $|x|\geq r/\sqrt{2}$ we have $\left|e^{x+\ie y}-1\right|\geq 1-e^{-r/\sqrt{2}}$ while for $|y|=r/\sqrt{2}$ we have $\left|e^{x+\ie y}-1\right|\geq \sin{\left(r/\sqrt{2}\right)}$. This gives us 
\[
\min_{z\in\partial\mathcal{S}_r}\left|e^z-1\right|\geq \min\left\{\sin{\frac{r}{\sqrt{2}}},1-e^{-\frac{r}{\sqrt{2}}}\right\}=1-e^{-\frac{r}{\sqrt{2}}}. 
\]
Take large $k\in\N$ and let $\mathcal{S}_k'$ be a two-dimensional closed square with vertices $-k\pm 2k\pi\ie$ and $k\pm 2k\pi\ie$. Define $\Omega(r,k)\de \left(\C\setminus\mathcal{S}_r\right)\cap \mathcal{S}_k'$. Then $\min_{z\in\partial\Omega(r,k)}\left|e^z-1\right|\geq 1-e^{-\frac{r}{\sqrt{2}}}$. Because the set $\Omega(r,k)$ is bounded and $e^z-1$ is holomorphic in the interior, the minimum principle implies $\left|e^z-1\right|\geq 1-e^{-\frac{r}{\sqrt{2}}}$ for every $z\in\Omega(r,k)$. Lemma \ref{lem:max} now follows because for every $z\in\C\setminus\mathcal{D}_r$ there exists $k$ such that $z\in\Omega(r,k)$.
\end{proof}

Numerical calculations suggest that the minimum value of $\left|e^z-1\right|$ on the set $\left\{z\in\C\colon |z|=r\right\}$ occurs at $z=-r$, thus giving lower bound $1-e^{-r}$ in \eqref{eq:M}. But the author is unable to prove this claim.

We are now in a position to prove Theorem \ref{thm:eafe}. We follow the proof presented in \cite{Titchmarsh} but with flexible parameters that have exactly prescribed domains of validity. This allows some optimisation when trying to get the best possible uniform bound for the remainder. 

\begin{theorem}
\label{thm:eafe}
Let $s=\sigma+\ie t$ where $\sigma\in[0,1]$ and $|t|>t_0\geq 2\pi$. Also assume that $2\pi xy=|t|$ for $x,y\geq 1$ where $x\leq y$. In addition, let $r_0$, $c$, $\lambda_0$ and $d$ be four real numbers satisfying the following conditions:  
\begin{itemize}
    \item[(a)] $0<r_0\leq\pi/\sqrt{2}$,
    \item[(b)] $\frac{r_0x}{|t|}\leq c\leq \frac{3\sqrt{2}}{10+4(1-\sigma)/t_0}$,
    \item[(c)] $r_0\leq\lambda_0\leq \frac{c|t|}{x}$,
    \item[(d)] $d\geq \frac{\pi x}{2\fl{x}}$.
\end{itemize}
Define four functions $E_1\left(\sigma,t,c,d,x\right)$, $E_2\left(\sigma,t,r_0,c,\lambda_0,x,y\right)$, $E_3\left(\sigma,t,c,x\right)$ and
\newline
$E_4\left(\sigma,t,r_0,c,x,y\right)$ in the following way:
\begin{equation*}
E_1 \de \mathcal{E}_1^{\frac{\sigma-1}{2}}\left(\sqrt{\frac{|t|}{\pi}}\mathcal{E}_2e^{-|t|\Phi_1}+\mathcal{E}_3e^{-|t|\Phi_2}\right), 
\end{equation*}
where
\begin{gather*}
\mathcal{E}_1(c)\de 2c^2+2c+1, \quad \mathcal{E}_2(t,c,d,x)\de d-c+\frac{x}{|t|}\log{\frac{1-e^{-\frac{d|t|}{x}}}{1-e^{-\frac{c|t|}{x}}}}, \\
\mathcal{E}_3(t,d,x)\de \frac{x}{\fl{x}\sqrt{\pi |t|}\left(e^{\frac{d|t|}{x}}-1\right)}, \\
\Phi_1(c) \de c-\arctan{\frac{c}{1+c}}, \quad \Phi_2(d,x) \de \frac{d\fl{x}}{x}-\frac{\pi}{2};
\end{gather*}
\begin{equation*}
E_2 \de \sqrt{\frac{2}{\mathcal{E}_4}}\left(\frac{1}{1-e^{-\frac{r_0}{\sqrt{2}}}}+\frac{1}{e^{\lambda_0}-1}\right)e^{\mathcal{E}_5}+\frac{r_0}{1-e^{-\frac{r_0}{\sqrt{2}}}}\sqrt{\frac{x}{2y}}\left(1+\frac{1}{y}\right)^{\sigma-1}e^{\mathcal{E}_6},
\end{equation*}
where
\begin{gather*}
\mathcal{E}_4\left(\sigma,t,c\right)\de 1-\frac{(1-\sigma+|t|)2\sqrt{2}c}{3|t|\left(1-c\sqrt{2}\right)}, \\ 
\mathcal{E}_5\left(\sigma,t,c,\lambda_0,x\right)\de \lambda_0\left(x-\fl{x}\right)+\frac{(1-\sigma)^2}{4|t|\mathcal{E}_4\left(\sigma,|t|,c\right)},
\end{gather*}
\begin{flalign*}
\mathcal{E}_6\left(\sigma,t,r_0,x,y\right)\de \frac{(1-\sigma)r_0}{2\pi\lfloor y\rfloor} &+\frac{r_0\left(xy-\lfloor y\rfloor\lfloor x\rfloor\right)}{\fl{y}} \\
&+\frac{r_0^2xy\left(1+\frac{1-\sigma}{|t|}\right)\left(\frac{1}{2}+\frac{r_0}{3(2\pi\lfloor y\rfloor-r_0)}\right)}{2\pi\lfloor y\rfloor^2};
\end{flalign*}
\begin{equation*}
E_3 \de c^{\sigma-1}\frac{2-c+\frac{\pi x}{|t|}}{1-e^{-\frac{c|t|}{x}}}\sqrt{\frac{|t|}{\pi}}e^{-|t|\Phi_1(-c)}, 
\end{equation*}
and
\begin{equation*}
E_4 \de \frac{x\left(1-\frac{\pi x}{|t|}\right)^{\sigma-1}}{\fl{x}\left(1-e^{-\frac{r_0}{\sqrt{2}}}\right)\sqrt{\pi |t|}}e^{-|t|\Phi_3},
\end{equation*}
where
\begin{gather*}
\Phi_3\left(t,c,x,y\right) \de \frac{\pi}{2} +  \arctan{\frac{1}{c}\left(1+\frac{\pi x(1-2\{y\})}{|t|}\right)}-\frac{\fl{x}}{x}c.     
\end{gather*}
Define $E(\sigma)\de 2^{-\sigma}C_2(t)\left(E_1+E_2+E_3+E_4\right)$ where $C_2(t)$ is defined by~\eqref{eq:C2}. Then we have
\begin{gather}
    \left|R_1\left(s;x,y\right)\right|\leq x^{-\sigma}E, \label{eq:err1} \\
    \left|R_1\left(s;y,x\right)\right|\leq \left(\frac{|t|}{2\pi}\right)^{\frac{1}{2}-\sigma}x^{\sigma-1}F \label{eq:err2}
\end{gather}
where $F=E$ if $\sigma=1/2$ or $x=y$, and
\begin{equation*}
\frac{F}{E\left(1-\sigma\right)}\leq \left\{
\begin{array}{ll}
    1+\frac{\left|C\left(\sigma,t,t_0\right)\right|}{|t|}; & \sigma\in(1/2,1], \\
     1+\frac{\left|C\left(1-\sigma,t,t_0\right)\right|}{|t|-\left|C\left(1-\sigma,t,t_0\right)\right|}; & \sigma\in[0,1/2).
\end{array}
\right.
\end{equation*}
Moreover, if $c,d,\lambda_0$ and $r_0$ are fixed, then $E$ is bounded and the parts $E_1,E_3,E_4$ are decreasing to zero while $t\to\infty$.
\end{theorem}

\begin{proof}
Firstly, we will show how to obtain \eqref{eq:err2} from \eqref{eq:err1}. Changing $s$ to $1-s$ in \eqref{eq:afe} and multiplying both sides by $\chi(s)$, we obtain the approximate functional equation with reversed role of $x,y$ and $R_1(s;y,x)=R_1(1-\sigma-\ie t;x,y)\chi(\sigma+\ie t)$.
This implies $\left|R_1\left(s;y,x\right)\right|\leq E(1-\sigma)\left|\chi(\sigma+\ie t)\right|x^{\sigma-1}$. Since $\left|\chi(1/2+\ie t)\right|=1$ our assertions for $\sigma\in[1/2,1]$ follow directly from Proposition \ref{prop:chi}, and for $\sigma\in[0,1/2)$ by inequality \eqref{eq:chiBoundRev}.

The main equation in the analytical proof of \eqref{eq:afe} is
\begin{equation}
\label{eq:mainafe}
\zeta(s)=\sum_{n=1}^{\fl{x}} n^{-s} + \chi(s)\sum_{n=1}^{\fl{y}} n^{s-1} + \frac{e^{-\ie\pi s}\Gamma(1-s)}{2\pi\ie}\int_{\mathcal{C}}\frac{z^{s-1}e^{-\fl{x}z}}{e^z-1}\dif{z}
\end{equation}
where $\mathcal{C}$ is a positively oriented contour $\mathcal{C}$ which goes from $+\infty$, encircles zeros $\pm2l\pi\ie$ of $e^z-1$ with $l\in\left\{0,1,2,\ldots,\fl{y}\right\}$, and returns back to $+\infty$, see \cite[pp.~99--100]{Ivic} for a detailed derivation of \eqref{eq:mainafe}. 

Let $\left[a,b\right]$ be a line segment in the complex plane with endpoints $a$ and $b$. Define $\eta\de 2\pi y=t/x$, $z_1\de c\eta+\ie\eta(1+c)$, $z_2\de-c\eta+\ie\eta(1-c)$ and $z_3\de-c\eta-\ie\pi\left(2\fl{y}+1\right)$. Also define $q$ as $\fl{y}$ if $\{y\}\leq1/2$ and $\fl{y}+1$ otherwise. The reader is advised to consult Figure \ref{fig:contour}. Because of the condition (a), the set
$\mathscr{I}\de\left[z_1,z_2\right] \cap \partial\mathcal{D}_{r_0}$, where $\mathcal{D}_r$ is defined in Lemma \ref{lem:max}, is empty or contains exactly two elements, say $w_1$ and $w_2$. Without loss of generality we can assign $w_1$ to the point closer to $z_1$. In the latter case, these two points are on the same circle with radius $r_0$ and center at $2\pi\ie q$, unless $r_0=\pi/\sqrt{2}$ and $\eta=\pi(2l+1)$. 

Deform $\mathcal{C}$ into four curves. Let $\mathcal{C}_1\de\left[\infty+\ie\eta(1+c),z_1\right]$, $\mathcal{C}_3\de\left[z_2,z_3\right]$ and $\mathcal{C}_4\de\left[z_3,\infty-\ie\pi\left(2\fl{y}+1\right)\right]$. Let $\mathcal{C}_{2}\de \left[z_1,z_2\right]$ if $\mathscr{I}=\emptyset$, and $\left[z_1,w_1\right]\cup \wideparen{w_1w_2}\cup\left[w_2,z_1\right]$ otherwise where $\wideparen{w_1w_2}$ is a smaller arc on circle if both points belong to the same circle. If this is not the case, we take the segment $\left[w_1,w_2\right]$ instead of the arc. Anyway, such contour always lies in $\C\setminus\mathcal{D}_{r_0}$.

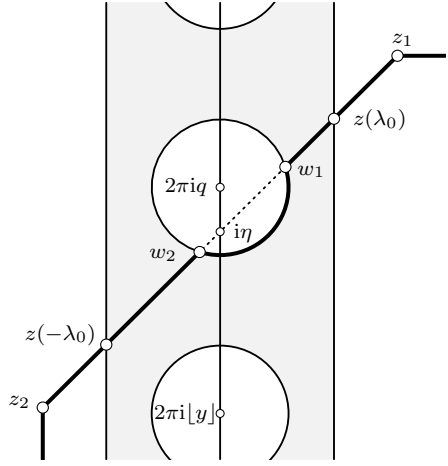
\begin{figure}[ht]
\centering
\begin{tikzpicture}[line cap=round,line join=round,>=triangle 45,x=1.0cm,y=1.0cm,scale=3]
\clip (-1.0221508960970906,0.7976971312841435) rectangle(0.988102235177704,2.8155685633639127);
\draw [gray!10,fill=gray!10] (0.49919716361905786,0.7976971312841435) rectangle(-0.49919716361905786,2.8155685633639127);
\draw [line width=0.7pt,fill=white] (0.,1.) circle (0.3cm);
\draw [white,fill=white] (0.,2.) circle (0.3cm);
\draw [line width=0.7pt,fill=white] (0.,3.) circle (0.3cm);
\draw [line width=0.7pt] (0.,0.7976971312841435) -- (0.,2.8155685633639127);
\draw [line width=2.pt,domain=-1.0221508960970906:1.001672959086163] plot(\x,{(-0.-0.*\x)/1.});
\draw [line width=0.7pt] (0.49919716361905786,0.7976971312841435) -- (0.49919716361905786,2.8155685633639127);
\draw [line width=0.7pt] (-0.49919716361905786,0.7976971312841435) -- (-0.49919716361905786,2.8155685633639127);
\draw [line width=1.5pt,domain=0.7774896848967332:1.001672959086163] plot(\x,{(--3.1543409803724454-0.*\x)/1.2220698048434757});
\draw [line width=1.5pt] (-0.7774896848967332,1.0261669508755542) -- (-0.7774896848967332,0.7976971312841435);
\draw [shift={(0.,2.)},line width=1.5pt]  plot[domain=-1.8750593429834028:0.30426301618850654,variable=\t]({1.*0.3*cos(\t r)+0.*0.3*sin(\t r)},{0.*0.3*cos(\t r)+1.*0.3*sin(\t r)});
\draw [line width=1.5pt] (0.7774896848967332,2.5811463206690206)-- (0.2862204014153603,2.089877037187648);
\draw [line width=1.5pt] (-0.08987703718764793,1.7137795985846398)-- (-0.7774896848967332,1.0261669508755542);
\draw [shift={(0.,2.)},line width=0.7pt]  plot[domain=0.3042630161885065:4.408125964196183,variable=\t]({1.*0.3*cos(\t r)+0.*0.3*sin(\t r)},{0.*0.3*cos(\t r)+1.*0.3*sin(\t r)});
\draw [line width=0.7pt,dotted] (0.2862204014153603,2.089877037187648)-- (-0.08987703718764793,1.7137795985846398);
\begin{footnotesize}
\draw (0.7,2.3) node {$z(\lambda_0)$};
\draw (-0.15,2) node {$2\pi\mathrm{i}q$};
\draw (-0.7,1.35) node {$z(-\lambda_0)$};
\draw (0.1,1.8) node {$\mathrm{i}\eta$};
\draw (-0.15,1) node {$2\pi\mathrm{i}\lfloor{y}\rfloor$};
\draw (0.8,2.65) node {$z_1$};
\draw (-0.88,1.05) node {$z_2$};
\draw (0.4,2.07) node {$w_1$};
\draw (-0.25,1.7) node {$w_2$};
\end{footnotesize}
\begin{scriptsize}
\draw [fill=white] (0.,1.) circle (0.5pt);
\draw [fill=white] (0.,2.) circle (0.5pt);
\draw [fill=white] (0.,1.8036566357722874) circle (0.5pt);
\draw [fill=white] (0.7774896848967332,2.5811463206690206) circle (0.7pt);
\draw [fill=white] (-0.7774896848967332,1.0261669508755542) circle (0.7pt);
\draw [fill=white] (0.49919716361905786,2.3028537993913454) circle (0.7pt);
\draw [fill=white] (-0.49919716361905786,1.3044594721532294) circle (0.7pt);
\draw [fill=white] (0.2862204014153603,2.089877037187648) circle (0.7pt);
\draw [fill=white] (-0.08987703718764793,1.7137795985846398) circle (0.7pt);
\end{scriptsize}
\end{tikzpicture}
\caption{Part of the contour integration (black thick line) along $\mathcal{C}$. The set composed by open circles with radii $r_0$ and centres at $2\pi\ie k$ is denoted by $\mathcal{D}_{r_0}$.}
\label{fig:contour}
\end{figure}

Write $z=u+\ie v=re^{\ie\varphi}$ where $r>0$. Then $\left|z^{s-1}\right|=r^{\sigma-1}e^{-\varphi t}$ and $\left|e^{-mz}\right|=e^{-mu}$. Denote by $I_k$ the integral in \eqref{eq:mainafe} which goes along $\mathcal{C}_k$. In the next paragraphs we will derive explicit bounds for each $I_k$ which will, together with Corollary \ref{cor:factor}, give the final bounds.

Consider integration along $\mathcal{C}_1$. We have
\begin{equation}
\label{eq:C1bound1}
\left|\frac{z^{s-1}e^{-\fl{x}z}}{e^z-1}\right| \leq \eta^\sigma \mathcal{E}_1(c)^{\frac{\sigma-1}{2}}\frac{1}{\eta}\frac{e^{-\fl{x}u}}{e^u-1},
\end{equation}
and also
\begin{equation}
\label{eq:C1bound2}
\left|\frac{z^{s-1}e^{-\fl{x}z}}{e^z-1}\right| \leq \eta^\sigma \mathcal{E}_1(c)^{\frac{\sigma-1}{2}} e^{-\frac{\pi}{2}t-t\Phi_1(c)}\frac{1}{\eta}\frac{e^u}{e^u-1}.   
\end{equation}
Note that $\Phi_1(c)$ is strictly increasing, thus $\Phi_1(c)>0$. The last inequality is true because
\[
-t\arctan{\frac{\eta(1+c)}{u}}-(\fl{x}+1)u \leq -t\arctan{\frac{\eta(1+c)}{u}}-\frac{tu}{\eta} \leq -\frac{\pi}{2}t-t\Phi_1(c)
\]
since the function in the middle is strictly decreasing in the variable $u$ and $\arctan{\alpha}+\arctan{1/\alpha}=\pi/2$ for $\alpha>0$. Let $d$ satisfies the condition (d). Then $d>c$ and 
\begin{flalign*}
\left|I_1\right| &\leq \left(\int_{c\eta}^{d\eta}+\int_{d\eta}^{\infty}\right) \left|\frac{z^{s-1}e^{-\fl{x}z}}{e^z-1}\right|\dif{u}\\
&\leq \eta^\sigma \mathcal{E}_1(c)^{\frac{\sigma-1}{2}}\left(e^{-\frac{\pi}{2}t-t\Phi_1(c)}\frac{1}{\eta}\log{\frac{e^{d\eta}-1}{e^{c\eta}-1}}+\frac{e^{-\fl{x}d\eta}}{\fl{x}\eta\left(e^{d\eta}-1\right)}\right)
\end{flalign*}
where we use \eqref{eq:C1bound2} for the first integral and \eqref{eq:C1bound1} for the second one. This implies that
\[
\left|\frac{e^{-\ie\pi s}\Gamma(1-s)}{2\pi\ie}I_1\right|\leq \frac{C_2(t)}{2^\sigma}E_1x^{-\sigma}.
\]
Note that $\Phi_1(c)>0$ and condition (d) imply that $E_1\to0$ while $t\to\infty$ if $c$ and $d$ are fixed.

Consider integration along $\mathcal{C}_2$. The main idea is to apply the bound from Lemma \ref{lem:max} on a part of $\mathcal{C}_2$ which goes through   
$\left\{z\in\C\colon \left|\Re\{z\}\right|\leq\lambda_0\right\}\setminus\mathcal{D}_{r_0}$,
where $\lambda_0$ satisfies the condition (c). This set is represented by the grey colour in Figure \ref{fig:contour}. Firstly, observe that for $|z|<1$ we can write $\log{(1+z)}=z+f_1(z)=z-z^2/2+z^3f_2(z)$ with 
\[
   |f_1(z)|\leq |z|^2\left(\frac{1}{2}+\frac{|z|}{3(1-|z|)}\right), \quad |f_2(z)|\leq \frac{1}{3(1-|z|)}.
\]
Let $z(\lambda)=\ie\eta+\lambda\sqrt{2}e^{\ie\pi/4}$, $\lambda\in[-c\eta,c\eta]$, be a parametrisation of the line $\left[z_1,z_2\right]$. Then we have
\begin{flalign*}
\log{z(\lambda)^{s-1}}-\log{\left(e^{(s-1)\frac{\pi}{2}\ie}\eta^{s-1}\right)}&=(s-1)\log{\left(1+\frac{\lambda\sqrt{2}}{\eta}e^{-\frac{\pi}{4}\ie}\right)} \\
&=(s-1)\left(\frac{\lambda\sqrt{2}}{\eta}e^{-\frac{\pi}{4}\ie}-\frac{\lambda^2}{\eta^2}e^{-\frac{\pi}{2}\ie}+\frac{\lambda^32\sqrt{2}}{\eta^3}e^{-\frac{3\pi}{4}\ie}o\right)
\end{flalign*}
where $|o|\leq \left(3\left(1-c\sqrt{2}\right)\right)^{-1}$. The above equation is valid if $|\lambda|<\eta/\sqrt{2}$, and this is true because $c<1/\sqrt{2}$ due to the condition (b). From this we obtain
\[
\left|z(\lambda)^{s-1}\right|\leq \eta^{\sigma-1}\exp{\left(t\left(-\frac{\pi}{2}+\frac{\lambda}{\eta}-\mathcal{E}_4(\sigma,t,c)\frac{\lambda^2}{\eta^2}+\frac{(\sigma-1)\lambda}{t\eta}\right)\right)}.
\]
Note that $\mathcal{E}_4>0$. Writing $e^{-\fl{x}z}=e^{z\left(x-\fl{x}\right)}e^{-zx}$ and noticing that $\cosh{(ax)}\geq\cosh{((1-a)x)}$ for $a\geq0$, we have
\[
\left|\frac{e^{-\fl{x}z(\lambda)}}{e^{z(\lambda)}-1}\right|\leq e^{-t\frac{\lambda}{\eta}}\frac{e^{\lambda_0(x-\fl{x})}}{e^{\lambda_0}-1}
\]
for $|\lambda|\geq \lambda_0$. Denote the integration along segments $\left[z_1,z\left(\lambda_0\right)\right]$ and $\left[z\left(-\lambda_0\right),z_2\right]$, and $\left[z\left(\lambda_0\right),w_1\right]$ and $\left[w_2,z\left(-\lambda_0\right)\right]$ by $I_{21}$ and $I_{22}$, respectively. Because
\[
I\left(a,b\right)\de \int_{-\infty}^{\infty}e^{-(a\lambda)^2+b\lambda}\dif{\lambda}=\frac{\sqrt{\pi}}{|a|}\exp{\left(\frac{b^2}{4a^2}\right)}
\]
for real numbers $a$ and $b$, see \cite[Eq.~\textbf{3.323} 2]{GradRyz}, it follows
\begin{flalign*}
\left|I_{21}\right| &\leq \sqrt{2}\eta^{\sigma-1}\frac{e^{\lambda_0(x-\fl{x})}}{e^{\lambda_0}-1}e^{-\frac{\pi}{2}t}I\left(\frac{\sqrt{t\mathcal{E}_4}}{\eta},\frac{\sigma-1}{\eta}\right) \\
&\leq t^{\sigma-\frac{1}{2}}x^{-\sigma} \sqrt{\frac{2\pi}{\mathcal{E}_4}}\frac{1}{e^{\lambda_0}-1}\exp{\left(-\frac{\pi}{2}t+\mathcal{E}_5\right)}.
\end{flalign*}
The bound for $I_{22}$ is the same except that we must replace $e^{\lambda_0}-1$ by $1-e^{-\frac{r_0}{\sqrt{2}}}$ in the above inequality.

Let $z=2\pi q\ie + r_0e^{\ie\varphi}$ be a parametrisation of the circle with center at $2\pi q\ie$ and radius $r_0$. Denote the integration along the arc $\wideparen{w_1w_2}$ by $I_{23}$.
Since
\begin{flalign*}
(s-1)\log{\left(1+\frac{r_0e^{\ie\varphi}}{2\pi q\ie}\right)}-r_0\fl{x}e^{\ie\varphi} &= \frac{(\sigma-1)r_0e^{\ie\varphi}}{2\pi q\ie}+(\sigma-1+\ie t)f_1\left(\frac{r_0e^{\ie\varphi}}{2\pi q\ie}\right) \\ &+\left(\frac{t}{2\pi q}-\fl{x}\right)r_0e^{\ie\varphi}, 
\end{flalign*}
we have
\[
\left|I_{23}\right| \leq \frac{r_0\pi}{1-e^{-\frac{r_0}{\sqrt{2}}}}t^{\sigma-\frac{1}{2}}x^{-\sigma}\sqrt{\frac{x}{2\pi y}}\left(1+\frac{1}{y}\right)^{\sigma-1}e^{-\frac{\pi}{2}t+\mathcal{E}_6}. 
\]
Because $I_2=I_{21}+I_{22}+I_{23}$, we finally obtain
\[
\left|\frac{e^{-\ie\pi s}\Gamma(1-s)}{2\pi\ie}I_2\right|\leq \frac{C_2(t)}{2^{\sigma}}E_2x^{-\sigma}.
\]
Note that $E_2$, although bounded for fixed $c$, $\lambda_0$ and $r_0$, does not tend to zero while $t\to\infty$ due to a contribution from parts $I_{21}$ and $I_{22}$.

Consider integration along $\mathcal{C}_3$. Because
\[
\left|\frac{z^{s-1}e^{-\fl{x}z}}{e^z-1}\right| \leq \frac{(c\eta)^{\sigma-1}}{1-e^{-c\eta}}e^{-\frac{\pi}{2}t-t\Phi_1(-c)},
\]
we have
\[
\left|I_3\right| \leq \frac{c^{\sigma-1}(2-c+\pi/\eta)}{1-e^{-c\eta}}\eta^{\sigma}e^{-\frac{\pi}{2}t-t\Phi_1(-c)}
\]
since $\eta(1-c)+\left(2\fl{y}+1\right)\pi\leq\eta(2-c+\pi/\eta)$. Note that $\Phi_1(-c)$ is strictly increasing, thus $\Phi_1(-c)>0$. From this we obtain
\[
\left|\frac{e^{-\ie\pi s}\Gamma(1-s)}{2\pi\ie}I_3\right|\leq \frac{C_2(t)}{2^\sigma}E_3x^{-\sigma}.
\]
Note that $\Phi_1(-c)>0$ implies that $E_3\to0$ while $t\to\infty$ if $c$ is fixed.

Consider integration along $\mathcal{C}_4$. Because $(2\fl{y}+1)\pi>\eta-\pi$, we have
\[
\left|z^{s-1}e^{-\fl{x}z}\right|\leq \frac{1}{\eta}\left(1-\frac{\pi}{\eta}\right)^{\sigma-1}\eta^{\sigma}\exp{\left(-\frac{\pi}{2}t-t\Phi_3-\frac{\fl{x}}{x}tc-\fl{x}u\right)}.
\]
Then
\begin{equation*}
\left|I_4\right|\leq
\frac{x}{1-e^{-\frac{r_0}{\sqrt{2}}}}\left(1-\frac{\pi x}{t}\right)^{\sigma-1}t^{\sigma-1}x^{-\sigma}e^{-\frac{\pi}{2}t}\left(\int_{-c\eta}^{\infty}e^{-\fl{x}u}\dif{u}\right)e^{-t\Phi_3-\frac{\fl{x}}{x}tc},
\end{equation*}
which gives
\[
\left|\frac{e^{-\ie\pi s}\Gamma(1-s)}{2\pi\ie}I_4\right|\leq \frac{C_2(t)}{2^\sigma}E_4x^{-\sigma}.
\]
Note that $E_4\to0$ while $t\to\infty$ since $c<\pi/2\leq \pi x/\left(2\fl{x}\right)$.
\end{proof}



\subsection{Numerical analysis of the error term.}
\label{sec:numerics}

Let $0\leq\sigma_0\leq\sigma\leq1$. Among all four terms in $E$, the $E_2$ is the only one which does not go asymptotically to zero, also because of term $\mathcal{E}_5$. This suggests we choose $\lambda_0$ as small as possible according to the condition (c) of Theorem \ref{thm:eafe}, therefore $\lambda_0=r_0$. Because $r_0x/|t|\leq r_0/(2\pi)\leq 1/\left(2\sqrt{2}\right)$, the choice $c=r_0/(2\pi)$ satisfies the condition (b). Putting these two parameters into $E_2$, we can obtain 
\begin{flalign*}
E_2\leq &\sqrt{\frac{6\left(\pi\sqrt{2}-r_0\right)}{3\pi\sqrt{2}-r_0\left(5+\frac{2\left(1-\sigma_0\right)}{|t|}\right)}}\left(\frac{1}{1-e^{-\frac{r_0}{\sqrt{2}}}}+\frac{1}{e^{r_0}-1}\right) \\
&\cdot\exp{\left(r_0+\frac{\left(1-\sigma_0\right)^2\left(6\pi-3r_0\sqrt{2}\right)}{4|t|\left(6\pi-5r_0\sqrt{2}\right)-8\left(1-\sigma_0\right)r_0\sqrt{2}}\right)}+\frac{r_0/\sqrt{2}}{1-e^{-\frac{r_0}{\sqrt{2}}}} \\
&\cdot\exp{\Biggl(\frac{r_0^2\left(|t|+1-\sigma_0\right)}{2\pi|t|}\left(\frac{1}{2}+\frac{r_0}{3\left(2\pi\fl{\sqrt{\frac{|t|}{2\pi}}}-r_0\right)}\right)\left(1+\fl{\sqrt{\frac{|t|}{2\pi}}}^{-1}\right)^2} \\
&+r_0\left(2+\fl{\sqrt{\frac{|t|}{2\pi}}}^{-1}\right)+\frac{\left(1-\sigma_0\right)r_0}{2\pi}\fl{\sqrt{\frac{|t|}{2\pi}}}^{-1}\Biggr).
\end{flalign*}
Taking $|t|\to\infty$ in the above expression, we get
\[
\sqrt{\frac{6\left(\pi\sqrt{2}-r_0\right)}{3\pi\sqrt{2}-5r_0}}\left(\frac{1}{1-e^{-\frac{r_0}{\sqrt{2}}}}+\frac{1}{e^{r_0}-1}\right)e^{r_0}+\frac{r_0/\sqrt{2}}{1-e^{-\frac{r_0}{\sqrt{2}}}}\exp{\left(\frac{r_0^2}{4\pi}+2r_0\right)}.
\]
Because this function is positive and continuous for $0<r_0\leq\pi/\sqrt{2}$ with a pole at $r_0=0$, it must have a minimum value on this interval. Let $R_0$ be the upper bound of the set where the minimum value is attained. Numerical calculations show that there is only one stationary point at $R_0\approx 0.52777$ and the minimum value is $\approx15.2029$. 

Choosing $d=\pi x/\left(2\fl{x}\right)<\pi$, and using inequalities $\pi x/|t|\leq\sqrt{\pi/(2|t|)}$ and $c|t|/x\geq r_0\sqrt{|t|/(2\pi)}$, we can estimate
\begin{flalign*}
E_1 \leq &\sqrt{\frac{|t|}{\pi}}\exp{\left(-|t|\left(\frac{r_0}{2\pi}-\arctan{\frac{r_0}{2\pi+r_0}}\right)\right)} \\
&\cdot\left(\pi-\frac{r_0}{2\pi}-\frac{1}{\sqrt{2\pi|t|}}\log{\left(1-e^{-r_0\sqrt{\frac{|t|}{2\pi}}}\right)}\right)+\frac{2}{\sqrt{\pi|t|}\left(e^{\frac{\pi}{2}\sqrt{2\pi|t|}}-1\right)}.
\end{flalign*}
Furthermore, we also have
\begin{flalign*}
E_3&\leq \left(\frac{r_0}{2\pi}\right)^{\sigma_0-1}\frac{\left(2-\frac{r_0}{2\pi}\right)\sqrt{\frac{2|t|}{\pi}}+1}{\left(1-e^{-r_0\sqrt{\frac{|t|}{2\pi}}}\right)\sqrt{2}}\exp{\left(-|t|\left(-\frac{r_0}{2\pi}+\arctan{\frac{r_0}{2\pi-r_0}}\right)\right)}, \\
E_4&\leq \left(1-\sqrt{\frac{\pi}{2|t|}}\right)^{\sigma_0-1}\frac{2}{\left(1-e^{-\frac{r_0}{\sqrt{2}}}\right)\sqrt{\pi|t|}}\exp{\left(-|t|\left(\frac{\pi}{2}-\frac{r_0}{2\pi}+\arctan{\frac{2\pi}{r_0}}\right)\right)}.
\end{flalign*}
From \cite[Theorems 4.1 and 4.2]{deReyna} we can deduce
\[
   \left|E_{\mathrm{L}}(s)\right| \leq \frac{1}{2} + \frac{9^{\sigma}}{2\sqrt{t}}+\left(\frac{11}{10}\right)^2\frac{2\pi}{7t}2^{\frac{3\sigma}{2}}
\]
for $\sigma\in(0,1]$ by taking the first two terms in $E_{\mathrm{L}}(s)$. Together with Proposition \ref{prop:chi} this implies
\[
\left|R_1\left(s;\sqrt{\frac{|t|}{2\pi}},\sqrt{\frac{|t|}{2\pi}}\right)\right|\leq \left(\frac{|t|}{2\pi}\right)^{-\frac{\sigma}{2}}\left(\left|E_{\mathrm{L}}(1)\right|+\left|E_{\mathrm{L}}\left(1-\sigma_0\right)\right|\left(1+\frac{0.3746}{|t|}\right)\right).
\]
Taking $r_0=0.52777$ and $\sigma_0=1/2$ in the above inequalities, we easily obtain bounds from Theorem \ref{thm:explmain}. After applying Proposition \ref{prop:chi} to Theorem \ref{thm:eafe}, we obtain \eqref{eq:explmaintild} with 
\[
\left|\widetilde{E}\right|\leq \frac{0.3746}{\sigma\sqrt{2\pi|t|}}+E, \quad 
\left|\widetilde{F}\right|\leq \frac{0.3746}{\sigma |t|}+F
\]
since $\sum_{n\leq X} n^{\sigma-1} \leq X^{\sigma}/\sigma$ is valid for $\sigma\in(0,1]$. This implies inequalities from Corollary \ref{cor:explmain}.

\subsection{Application to the approximate functional equation for $\zeta^2(s)$.} 

Hardy and Littlewood proved in \cite{HLapprox2} that
\begin{equation}
\label{eq:afesq}
\zeta^2(s) = \sum_{n\leq x}\frac{d(n)}{n^s} + \chi^2(s)\sum_{n\leq y}\frac{d(n)}{n^{1-s}} + R_2(s;x,y)
\end{equation}
with $4\pi^2 xy=t^2$, where $R_2(s;x,y)\ll x^{1/2-\sigma}\left((x+y)/|t|\right)^{1/4}\log{|t|}$. Here $d(n)$ is the divisor function, and it is well-known that $\sum_{n\leq X}d(n)=X\log{X}+(2\gamma-1)X+\Delta(X)$ where $\Delta(X)\ll \sqrt{X}$. Later Titchmarsh provided a different proof of \eqref{eq:afesq} with $R_2(s;x,y)\ll x^{1/2-\sigma}\log{|t|}$, see also \cite[pp.~104--121]{Ivic}. Both proofs are quite elaborate. In the symmetric case $x=y=|t|/(2\pi)$, Motohashi \cite{Motohashi} found a simple connection between \eqref{eq:afesq} and \eqref{eq:afe} by means of Dirichlet's hyperbola method. He obtained
\begin{flalign*}
R_2(s)&\de R_2\left(s;\frac{|t|}{2\pi},\frac{|t|}{2\pi}\right) = 2\chi(s)\sum_{n\leq \sqrt{\frac{|t|}{2\pi}}}\frac{1}{n} + 2\sum_{n\leq \sqrt{\frac{|t|}{2\pi}}}\frac{R_1\left(s;\frac{|t|}{2\pi n},n\right)}{n^s} \\ 
&+ 2\chi^2(s)\sum_{n\leq \sqrt{\frac{|t|}{2\pi}}}\frac{R_1\left(1-s;\frac{|t|}{2\pi n},n\right)}{n^{1-s}} + R_1^2\left(s;\sqrt{\frac{|t|}{2\pi}},\sqrt{\frac{|t|}{2\pi}}\right),
\end{flalign*}
where $R_1(s;x,y)$ is the error term in the approximate functional equation. Theorem \ref{thm:eafe} enables us to obtain an explicit version of $R_2(s)$ and thus of \eqref{eq:afesq} in the symmetric case.

\begin{corollary}
\label{cor:eafesq}
Let $s=\sigma+\ie t$ where $\sigma\in[1/2,1]$ and $|t|\geq 10^3$. Then we have
\begin{equation}
\label{eq:afesqchi}
\zeta^2\left(s\right) = \sum_{n\leq \frac{|t|}{2\pi}}\frac{d(n)}{n^{s}} + \sgn{t}\ie \left(\frac{|t|}{2\pi}\right)^{1-2\sigma} \left(\frac{|t|}{2\pi e}\right)^{-2\ie t}\sum_{n\leq \frac{|t|}{2\pi}}\frac{d(n)}{n^{1-s}} + \widetilde{R}_2\left(s\right),
\end{equation}
where
\[
\left|\widetilde{R}_2\left(s\right)\right| \leq 34.765\left(\frac{|t|}{2\pi}\right)^{\frac{1}{2}-\sigma}\log{\frac{|t|}{2\pi}},
\] 
and also
\[
\left|\widetilde{R}_2\left(\frac{1}{2}+\ie t\right)\right| \leq 28.621\log{\frac{|t|}{2\pi}}.
\]
\end{corollary}

\begin{proof}
By symmetry, we can assume that $t\geq 10^3$. By Proposition \ref{prop:chi}, \eqref{eq:tildechi} and \eqref{eq:Cbound} we have $\left|\widetilde{R}_2(s)\right| \leq \left|R_2(s)\right| + r_2(s)$, where
\[
r_2(s) \de \left(\frac{t}{2\pi}\right)^{1-2\sigma}\frac{0.75}{t}\sum_{n\leq \frac{t}{2\pi}}\frac{d(n)}{n^{1-\sigma}}.
\]
Partial summation assures that
\begin{flalign*}
\sum_{n\leq X}\frac{d(n)}{n^{1-\sigma}} = \frac{1}{\sigma}X^{\sigma}\log{X} &+ \frac{2\gamma\sigma-1}{\sigma^2}X^{\sigma} + \frac{(2\gamma-1)\sigma^2-2\gamma\sigma+1}{\sigma^2} \\ 
&+ (1-\sigma)\int_{1}^{X}\frac{\Delta(u)}{u^{2-\sigma}}\dif{u} + X^{\sigma-1}\Delta(X)
\end{flalign*}
for $X\geq1$. Now we consider two cases: $\sigma=1/2$ and $\sigma\in(1/2,1]$. Using the elementary bound $\Delta(X)\leq 3\sqrt{X}$, we obtain
\[
r_2\left(\frac{1}{2}+\ie t\right) \leq 0.106, \quad r_2(s) \leq 0.265\left(\frac{t}{2\pi}\right)^{1-2\sigma}\log{\frac{t}{2\pi}}.
\]
There exist much better estimates for $\Delta(X)$, see \cite[Theorem 1.1]{BBR}, but this bound is good enough for our purposes.

Now we need to bound $R_2(s)$. Because $n\leq\sqrt{t/(2\pi)}$, this implies $t/(2\pi n)\geq n$. By Theorem \ref{thm:explmain} and Theorem \ref{thm:eafe} we thus have 
\begin{gather*}
\left|R_1\left(s;\frac{t}{2\pi n},n\right)\right| \leq 15.726\left(\frac{t}{2\pi}\right)^{\frac{1}{2}-\sigma}n^{\sigma-1}, \\
\left|R_1\left(1-s;\frac{t}{2\pi n},n\right)\right| \leq 10.988\left(\frac{t}{2\pi}\right)^{\sigma-\frac{1}{2}}n^{-\sigma}, \\
\left|R_1\left(\frac{1}{2}+\ie t;\frac{t}{2\pi n},n\right)\right|\leq \frac{10.983}{\sqrt{n}}.
\end{gather*}
Using also $\left|\chi(s)\right|\leq 1.00038\left(t/(2\pi)\right)^{1/2-\sigma}$, $\left|\chi(1/2+\ie t)\right|=1$, and the inequality 
\[
\sum_{n\leq X}\frac{1}{n} \leq \log{X}+\gamma+\frac{1}{2X}, 
\]
see \cite[Lemma 2.8]{DHA}, we obtain
\begin{gather*}
\left|R_2(s)\right| \leq 34.5\left(\frac{t}{2\pi}\right)^{\frac{1}{2}-\sigma}\log{\frac{t}{2\pi}}, \\
\left|R_2\left(\frac{1}{2}+\ie t\right)\right| \leq 28.6\log{\frac{t}{2\pi}}.
\end{gather*}
These bounds give the desired estimates from Corollary \ref{cor:eafesq}.
\end{proof}

In a work in progress we will use Corollary \ref{cor:eafesq} to obtain an explicit fourth power moment of the Riemann zeta-function which will be useful to get an explicit version of \eqref{eq:Ingham}. 

\section{Explicit second power moment of the Riemann zeta-function}
\label{sec:espm}

The main analytic tool used by Selberg in his proof of the zero density estimate is a weighted second power moment of $\zeta$, see \cite[Lemma 6]{SelbergContrib}. The main idea is to use the approximate functional equation in the form \eqref{eq:chitilde}, together with \eqref{eq:fafeg} for real values. In the forthcoming subsections we will provide a proof of the following explicit version of Selberg's lemma with $H=T$.

\begin{theorem}
\label{thm:SelbergMoment}
Let $\sigma\in\left(1/2,\sigma_0\right]$, $\sigma_0\in(1/2,1)$, and $T\geq T_0\geq 2\pi$. Furthermore, let $1\leq\mu_1\leq\mu_2\leq T/(2\pi)$ be a positive coprime integers, and denote $z\de\left(\sigma,T,\mu_1,\mu_2\right)$. Define
\[
S\left(z\right) \de \int_{T}^{2T} \left|\zeta(\sigma+\ie t)\right|^2\left(\frac{\mu_1}{\mu_2}\right)^{\ie t}\dif{t}
\]
and
\begin{equation}
\label{eq:selbergS}
\mathscr{S}\left(z\right) \de \frac{\zeta(2\sigma)}{\left(\mu_1\mu_2\right)^{\sigma}}T + \frac{(2\pi)^{2\sigma-1}\left(4^{1-\sigma}-1\right)\zeta(2-2\sigma)}{2(1-\sigma)\left(\mu_1\mu_2\right)^{1-\sigma}}T^{2(1-\sigma)}.
\end{equation}
Then
\[
\left|S - \mathscr{S}\right| \leq \mathscr{S}_1(z)\left(\frac{\mu_2}{\mu_1}\right)^{\sigma}T^{1-\frac{\sigma}{2}}\sqrt{\log{\frac{T\mu_2}{\pi\mu_1}}} + \mathscr{S}_2(z)\mu_1\mu_2 T^{1-\sigma}\log{\frac{T\mu_1\mu_2}{\pi}},
\]
where
\begin{flalign*}
\mathscr{S}_1(z) &\de \sqrt{\frac{\log{T_0}}{\log{\frac{T_0}{\pi}}}}\left(\mathcal{B}_5(z)+\mathcal{B}_6(z)\right) + \mathcal{B}_7(z)+\mathcal{B}_8(z), \\  
\mathscr{S}_2(z) &\de \mathcal{B}_1\left(\sigma_0,T_0\right) + \sqrt{\frac{1}{\pi}}\mathcal{B}_3(z) + \mathcal{B}_4(z) + \frac{\mathcal{B}_9\left(\sigma_0\right)}{\log{\frac{T_0}{\pi}}},
\end{flalign*}
and positive functions $\mathcal{B}_1,\mathcal{B}_3,\ldots,\mathcal{B}_9$, defined by equations \eqref{eq:b1}, \eqref{eq:b3}, \eqref{eq:b4}, \eqref{eq:b5}, \eqref{eq:b6}, \eqref{eq:b7}, \eqref{eq:b8}, and \eqref{eq:b9}, respectively, are bounded for fixed $\sigma_0$ and $T_0$. Additionally, they are continuous for $\sigma\in\left[1/2,\sigma_0\right]$ and $\sigma_0\in[1/2,1)$.
\end{theorem}

Although $\mathscr{S}(z)$ is not defined for $\sigma=1/2$, the limit $\sigma\to1/2$ exists. This enables us to obtain an explicit upper bound for second power moment of $\zeta$ on the critical line, see Corollary \ref{cor:espm}. It turns out that we get an explicit version of Littlewood's bound
\begin{equation*}
\int_{0}^{T} \left|\zeta\left(\frac{1}{2}+\ie t\right)\right|^2\dif{t} = T\log{T} - \left(1+\log{2\pi}-2\gamma\right)T + \mathscr{E}(T)
\end{equation*}
with $\mathscr{E}(T)=O\left(T^{\frac{3}{4}+\varepsilon}\right)$, announced incorrectly and without proof in 1922. This estimate was the first improvement of the fact that the integral is asymptotically equal to $T\log{T}$, a result due to Hardy and Littlewood. Their second proof uses the approximate functional equation, see \cite[Theorem 7.3]{Titchmarsh}. New turn in the mean square theory was Atkinson's formula for $\mathscr{E}(T)$ and its various generalisations, e.g., Matsumoto--Meurman formulas. They enabled to prove that $\mathscr{E}(T)=O\left(T^{\frac{35}{108}+\varepsilon}\right)$ and it is plausible to believe that $\mathscr{E}(T)=O\left(T^{\frac{1}{4}+\varepsilon}\right)$ is true since $\mathscr{E}(T)=\Omega\left(T^{\frac{1}{4}}\right)$, see \cite[Chapter 15]{Ivic} for proofs and techniques, and \cite{MatsumotoRecent} for an overview of the mean square theory. 

\begin{corollary}
\label{cor:espm}
Let $T\geq2\pi$. Then
\begin{equation}
\label{eq:spm}
\mathscr{E}(T) \leq 13.803T^{\frac{3}{4}}\sqrt{\log{\frac{T}{2\pi}}} +83.964\sqrt{T}\log{\frac{T}{2\pi}} + 2\cdot10^3\log{T} + 3691.24.
\end{equation}
\end{corollary}

\begin{proof}
Let $T_0=10^3$ and $T\geq 2T_0$. Define $S(\sigma,T)\de S\left(\sigma,T,1,1\right)$, $\mathscr{S}\left(\sigma,T\right)\de\mathscr{S}\left(\sigma,T,1,1\right)$, $\mathscr{S}_1\left(\sigma,T\right)\de\mathscr{S}_1\left(\sigma,T,1,1\right)$ and $\mathscr{S}_2\left(\sigma,T\right)\de\mathscr{S}_2\left(\sigma,T,1,1\right)$. Take an arbitrary $\sigma_0\in(1/2,1)$ and let $\sigma\in\left(1/2,\sigma_0\right]$. By Theorem \ref{thm:SelbergMoment} there exist a continuous functions $\widehat{\mathscr{S}}_1\left(\sigma_0,T_0\right)$ and $\widehat{\mathscr{S}}_2\left(\sigma_0,T_0\right)$ for $\sigma_0\in[1/2,1)$ such that $\mathscr{S}_1\left(\sigma,T\right)\leq \widehat{\mathscr{S}}_1\left(\sigma_0,T_0\right)$ and $\mathscr{S}_2\left(\sigma,T\right)\leq \widehat{\mathscr{S}}_2\left(\sigma_0,T_0\right)$. Also $\widehat{\mathscr{S}}_1(1/2,T_0)\leq 9.4104$ and $\widehat{\mathscr{S}}_2(1/2,T)\leq 34.779$. We thus have
\begin{flalign*}
\int_{2T_0}^{T}\left|\zeta\left(\sigma+\ie t\right)\right|^2\dif{t}&\leq \sum_{n=1}^{n_0} S\left(\sigma,\frac{T}{2^n}\right) \leq \widehat{\mathscr{S}}\left(\sigma\right) \\
&+ \frac{\widehat{\mathscr{S}}_1\left(\sigma_0,T_0\right) T^{1-\frac{\sigma}{2}}}{2^{1-\frac{\sigma}{2}}-1}\sqrt{\log{\frac{T}{2\pi}}}+\frac{\widehat{\mathscr{S}}_2\left(\sigma_0,T_0\right) T^{1-\sigma}}{2^{1-\sigma}-1}\log{\frac{T}{2\pi}},
\end{flalign*}
where $n_0\de\fl{\log_2{\left(T/T_0\right)}}$ and
\[
\widehat{\mathscr{S}}\left(\sigma\right) \de \sum_{n=1}^{n_0} \mathscr{S}\left(\sigma,\frac{T}{2^n}\right).
\] 
A simple calculation shows that
\[
\widehat{\mathscr{S}}\left(\sigma\right) = \zeta(2\sigma)\left(1-2^{-n_0}\right)T + f(\sigma)\zeta(2-2\sigma),
\]
where
\[
f(\sigma)\de \frac{(2\pi)^{2\sigma-1}\left(1-4^{-(1-\sigma)n_0}\right)}{2(1-\sigma)}T^{2(1-\sigma)}.
\]
Remember that the Laurent series of $\zeta(s)$ around $s=1$ is $\zeta(s)=(s-1)^{-1}+\gamma+g(s)$ for some holomorphic function $g(s)$ with $g(1)=0$, and $\gamma$ is the Euler--Mascheroni constant. Then 
\[
\widehat{\mathscr{S}}\left(\sigma\right) = \frac{\left(1-2^{-n_0}\right)T-f(\sigma)}{2\sigma-1} + \left(\gamma+g(2\sigma)\right)\left(\left(1-2^{-n_0}\right)T+f(\sigma)\right).
\]
Since $\lim_{\sigma\to1/2}f(\sigma)=\left(1-2^{-n_0}\right)T$, we have
\begin{flalign*}
\lim_{\sigma\to\frac{1}{2}}\widehat{\mathscr{S}}\left(\sigma\right) &= -\frac{1}{2}f'\left(\frac{1}{2}\right) + 2\gamma\left(1-2^{-n_0}\right)T = \left(1-2^{-n_0}\right)\left(\log{T}+2\gamma\right)T \\ 
&- \left(1+\log{2\pi}\right)T + 2^{-n_0}\left(1+n_0\log{2}+\log{2\pi}\right)T \\
&\leq T\log{T} - \left(1+\log{2\pi}-2\gamma\right)T + 2T_0\log{\frac{2\pi eT}{T_0}}.
\end{flalign*}
Take $\sigma_0\to1/2$. For $T\geq 2T_0$, the main inequality now easily follows from this since we can numerically verify that $\int_{0}^{2T_0}\left|\zeta(1/2+\ie t)\right|^2\dif{t}\leq 11831$, and this also implies that it is true for $T\geq35$. Finally $\int_{0}^{35}\left|\zeta(1/2+\ie t)\right|^2\dif{t}\leq 67$, which concludes the proof.
\end{proof}

Evaluation of the above integrals was performed in \emph{Mathematica}, using the built-in function \texttt{RiemannSiegelZ[t]} and integration method \texttt{NIntegrate}. In principle, it is possible to improve the constants in the first two terms in \eqref{eq:spm} because $\lim_{T\to\infty}\mathscr{S}_1\left(1/2,T,1,1\right)\approx8.953$ and $\lim_{T\to\infty}\mathscr{S}_2\left(1/2,T,1,1\right)\approx22.6$, but unfortunately $T_0$, and consequently the last two constants in \eqref{eq:spm}, grow too rapidly to be numerically useful. Note that our estimate is for $T\geq 1545$ better than the recent explicit bound in \cite[Theorem 4.3]{DHA}.

\subsection{Setting of the proof.}

Assume the conditions of Theorem \ref{thm:SelbergMoment}. Define 
\[
x(t)\de \sqrt{\frac{t\mu_1}{2\pi\mu_2}} \quad \textrm{and} \quad y(t)\de \sqrt{\frac{t\mu_2}{2\pi\mu_1}}.
\]
Then $1\leq x(t)\leq y(t)$ and $2\pi x(t)y(t)=t$. Using Corollary \ref{cor:explmain}, we obtain
\begin{multline*}
\zeta(\sigma+\ie t) = \sum_{n\leq x(t)}n^{-\sigma-\ie t} + \widetilde{\chi}(\sigma+\ie t) \sum_{n\leq y(t)}n^{\sigma-1+\ie t} \\
+ \left(\frac{t}{2\pi}\right)^{-\frac{\sigma}{2}}\left(\frac{\mu_1}{\mu_2}\right)^{-\frac{\sigma}{2}}\widetilde{E}(s;x(t),y(t))
\end{multline*}
with $\left|\widetilde{E}(s;x(t),y(t))\right|\leq \widetilde{E}$. Changing roles of $x(t)$ and $y(t)$, Corollary \ref{cor:explmain} also implies
\begin{multline*}
\zeta(\sigma-\ie t) = \sum_{m\leq y(t)}m^{-\sigma+\ie t} + \widetilde{\chi}(\sigma-\ie t) \sum_{m\leq x(t)}m^{\sigma-1-\ie t} \\ 
+ \left(\frac{t}{2\pi}\right)^{-\frac{\sigma}{2}}\left(\frac{\mu_1}{\mu_2}\right)^{\frac{\sigma-1}{2}}\widetilde{F}(s;x(t),y(t))
\end{multline*}
with $\left|\widetilde{F}(s;x(t),y(t))\right|\leq \widetilde{F}$. If we multiply these two equations, we get an expression for $\left|\zeta(s)\right|^2$ consisting of nine terms and arranged into five groups:
\begin{gather*}
A_1\de \sum_{n\leq x(t)}\sum_{m\leq y(t)}(nm)^{-\sigma}\left(\frac{m}{n}\right)^{\ie t}, \\
A_2\de \left(\frac{t}{2\pi}\right)^{1-2\sigma} \sum_{n\leq y(t)}\sum_{m\leq x(t)}(nm)^{\sigma-1}\left(\frac{n}{m}\right)^{\ie t};
\end{gather*}
\begin{gather*}
A_3\de \widetilde{\chi}(\sigma+\ie t)\mathop{\sum\sum}_{n,m\leq y(t)} n^{\sigma-1}m^{-\sigma}(nm)^{\ie t}, \\
A_4\de \widetilde{\chi}(\sigma-\ie t)\mathop{\sum\sum}_{n,m\leq x(t)} n^{\sigma-1}m^{-\sigma}(nm)^{-\ie t};
\end{gather*}
\begin{gather*}
A_5\de \left(\frac{t}{2\pi}\right)^{-\frac{\sigma}{2}}\left(\frac{\mu_1}{\mu_2}\right)^{\frac{\sigma-1}{2}}\widetilde{F}(s;x(t),y(t))\sum_{n\leq x(t)}n^{-\sigma-\ie t}, \\
A_6\de \left(\frac{t}{2\pi}\right)^{-\frac{\sigma}{2}}\left(\frac{\mu_1}{\mu_2}\right)^{-\frac{\sigma}{2}}\widetilde{E}(s;x(t),y(t))\sum_{n\leq y(t)}n^{-\sigma+\ie t};
\end{gather*}
\begin{gather*}
A_7\de \left(\frac{t}{2\pi}\right)^{-\frac{\sigma}{2}}\left(\frac{\mu_1}{\mu_2}\right)^{\frac{\sigma-1}{2}}\widetilde{F}(s;x(t),y(t))\widetilde{\chi}\left(\sigma+\ie t\right)\sum_{n\leq y(t)}n^{\sigma-1+\ie t}, \\
A_8\de \left(\frac{t}{2\pi}\right)^{-\frac{\sigma}{2}}\left(\frac{\mu_1}{\mu_2}\right)^{-\frac{\sigma}{2}}\widetilde{E}(s;x(t),y(t))\widetilde{\chi}\left(\sigma-\ie t\right)\sum_{n\leq x(t)}n^{\sigma-1-\ie t};
\end{gather*}
\begin{equation*}
A_9\de \left(\frac{t}{2\pi}\right)^{-\sigma}\sqrt{\frac{\mu_2}{\mu_1}}\widetilde{E}(s;x(t),y(t))\widetilde{F}(s;x(t),y(t)).
\end{equation*}
Therefore, 
\[
   S(\sigma,T;\mu_1,\mu_2) = \sum_{j=1}^9 \int_{T}^{2T} A_j\cdot\left(\frac{\mu_1}{\mu_2}\right)^{\ie t}\dif{t}.
\]
Denote by $B_i$ the $i$th summand in the above equation. In the following subsections we provide an explicit bounds on each $B_i$. Before doing this we firstly collect some lemmas which are used in the forthcoming subsections.

\subsection{Some lemmas.} The first lemma is a rule for changing integration and summation when the range in the sum depends on the integration variable.

\begin{lemma}
\label{lem:intsum}
Let $f(n,t)$ be an integrable function in variable $t\in\left[T_1,T_2\right]$ where $T_1\geq1$, and let $g(t)$ be a strictly increasing differentiable function with $g\left(T_1\right)\geq1$. Then
\[
\int_{T_1}^{T_2} \sum_{n\leq g(t)}f(n,t)\dif{t} = \sum_{n\leq g\left(T_2\right)}\int_{\max\left\{T_1,g^{-1}(n)\right\}}^{T_2}f(n,t)\dif{t}.
\]
\end{lemma}

\begin{proof}
We first prove the special case when $g(t)=t$. We can assume that $\fl{T_2}-\fl{T_1}\geq2$ since otherwise the lemma is obviously true. Then
\[
\int_{T_1}^{T_2} \sum_{n\leq t}f(n,t)\dif{t} = \left(\sum_{n\leq T_1}\int_{T_1}^{\fl{T_1}+1} + \sum_{j=\fl{T_1}+1}^{\fl{T_2}-1}\sum_{n\leq j}\int_{j}^{j+1}+\sum_{n\leq\fl{T_2}}\int_{\fl{T_2}}^{T_2}\right)f(n,t)\dif{t}.
\]
The second integral equals to
\[
\sum_{n\leq\fl{T_2}-1}\sum_{j=\max\left\{\fl{T_1}+1,n\right\}}^{\fl{T_2}-1}\int_{j}^{j+1}f(n,t)\dif{t}
=\sum_{n\leq\fl{T_2}-1}\int_{\max\left\{\fl{T_1}+1,n\right\}}^{\fl{T_2}}f(n,t)\dif{t}.
\]
This implies
\[
\int_{T_1}^{T_2} \sum_{n\leq t}f(n,t)\dif{t} = \left(\sum_{n\leq T_1}\int_{T_1}^{T_2}+\sum_{n=\fl{T_1}+1}^{\fl{T_2}}\int_{n}^{T_2}\right)f(n,t)\dif{t}
\]
and consequently the lemma in this special case.

Write $g(t)=u$. Because $g(t)$ is strictly increasing differentiable function, there exists its inverse $t=g^{-1}(u)$ and $\dif{t}=\left(g^{-1}(u)\right)'\dif{u}$. We have
\begin{flalign*}
\int_{T_1}^{T_2} \sum_{n\leq g(t)}f(n,t)\dif{t} &= \int_{g\left(T_1\right)}^{g\left(T_2\right)} \sum_{n\leq u}f\left(n,g^{-1}(u)\right)\left(g^{-1}(u)\right)'\dif{u} \\
&= \sum_{n\leq g\left(T_2\right)}\int_{\max\left\{g\left(T_1\right),n\right\}}^{g\left(T_2\right)}f\left(n,g^{-1}(u)\right)\left(g^{-1}(u)\right)'\dif{u} \\
&= \sum_{n\leq g\left(T_2\right)}\int_{g^{-1}\left(\max\left\{g\left(T_1\right),n\right\}\right)}^{T_2}f(n,t)\dif{t},
\end{flalign*}
where the second equality follows from the first part of the proof. Clearly, the last integral equals to the second integral from the lemma.
\end{proof}

From Lemma \ref{lem:intsum} it easily follows that
\begin{equation}
\label{eq:intsum}
\int_{T_1}^{T_2} \sum_{n\leq g_1(t)}\sum_{m\leq g_2(t)}f(n,m,t)\dif{t} = \sum_{n\leq g_1\left(T_2\right)}\sum_{m\leq g_2\left(T_2\right)}\int_{\mathcal{M}(n,m)}^{T_2}f(n,m,t)\dif{t},
\end{equation}
where $\mathcal{M}(n,m)\de\max\left\{T_1,g_1^{-1}(n),g_2^{-1}(m)\right\}$, and functions $f,g_1,g_2$ satisfy conditions of Lemma \ref{lem:intsum}.

The next two lemmas bound particular double sums which appear during the integration of Dirichlet polynomials. The first one is a slightly modified corollary of Preissmann's inequality
\begin{equation}
\label{eq:preiss}
\left|\mathop{\sum\sum}_{\substack{n,m\leq X\\ n\neq m}}\frac{u_n\overline{u_m}}{x_n-x_m}\right|\leq \pi m_0\sum_{n\leq X}\frac{\left|u_n\right|^2}{\min_{n\neq m}\left|x_n-x_m\right|},
\end{equation}
where $X\geq2$, $\left\{x_n\right\}_{n\leq X}$ are distinct real numbers, $\left\{u_n\right\}_{n\leq X}$ are complex numbers and $m_0\de\sqrt{1+\frac{2}{3}\sqrt{\frac{6}{5}}}$, see \cite{Preissmann}.

\begin{lemma}
\label{lem:preiss}
Let $X\geq2$, $\{a_n\}_{n\leq X}$ be a sequence of complex numbers and $Y\in\R$. Then
\[
\left|\mathop{\sum\sum}_{\substack{n,m\leq X\\ n\neq m}}\frac{a_n\overline{a_m}}{\log{(n/m)}}\left(\frac{n}{m}\right)^{\ie Y}\right|\leq \pi m_0\sum_{n\leq X}\left|a_n\right|^2\left(\frac{1}{2}+n\right),
\]
where $m_0\de\sqrt{1+\frac{2}{3}\sqrt{\frac{6}{5}}}$.
\end{lemma}

\begin{proof}
Use \eqref{eq:preiss} for $x_j=\log{j}$ and $u_j=a_j j^{\ie Y}$, and observe that $\left|\log{(n/m)}\right|\geq (n+1/2)^{-1}$ for all distinct integers $n$ and $m$.
\end{proof}

\begin{lemma}
\label{lem:doublesum}
Let $X\geq2$ and $|a|<1$. Then
\[
\mathop{\sum\sum}_{\substack{n,m\leq X\\ n\neq m}}\frac{(nm)^{a}}{\left|\log{(n/m)}\right|} \leq 
\left(\sum_{n\leq X}n^{a}\right)^2 - \sum_{n\leq X}n^{2a} + 2\sum_{n\leq X}n^{1+2a}\sum_{n\leq X}\frac{1}{n}.
\]
\end{lemma}

\begin{proof}
Firstly, observe that $1/\log{\lambda}\leq 1+\lambda^{1+a}/\left(\lambda-1\right)$ is true for $\lambda>1$. Then
\[
\mathop{\sum\sum}_{\substack{n,m\leq X\\ n\neq m}}\frac{(nm)^{a}}{\left|\log{(n/m)}\right|} \leq 
\mathop{\sum\sum}_{\substack{n,m\leq X\\ n\neq m}} (nm)^a + 
2\mathop{\sum\sum}_{n<m\leq X}\frac{m^{1+2a}}{m-n}
\]
from which the main inequality follows.
\end{proof}

The idea is to combine Lemmas \ref{lem:preiss} and \ref{lem:doublesum} to bound the following double sum 
\begin{equation*}
D\left(a,T_1,T_2;X\right) \de \mathop{\sum\sum}_{\substack{n,m\leq X\\ n\neq m}}\frac{(nm)^{a}}{\log{(n/m)}}\left(\left(\frac{n}{m}\right)^{\ie T_2}-\left(\frac{n}{m}\right)^{\ie T_1}\right).
\end{equation*}
In the most subsequent applications, $T_2$ is independent while $T_1$ depends on the summation variables $n$ and $m$. Thus we will use Lemma \ref{lem:preiss} for $a_n=n^a,Y=T_2$ to bound the first part, and Lemma \ref{lem:doublesum} for the second part. What we obtain is
\begin{multline}
\label{eq:doublesumbound}
\left|D\left(a,T_1,T_2;X\right)\right| \leq \left(\pi m_0 + 2\sum_{n\leq X}\frac{1}{n}\right)\sum_{n\leq X}n^{1+2a} \\ +\left(\sum_{n\leq X}n^{a}\right)^2
+ \left(\frac{\pi m_0}{2}-1\right)\sum_{n\leq X}n^{2a}.
\end{multline}
Observe that $\pi m_0/2-1>0$. We will need \eqref{eq:doublesumbound} only for $a\in\{-\sigma,\sigma-1\}$. Particular sums are estimated by
\begin{gather}
\sum_{n\leq X}n^{-2\sigma} \leq \log{X}+\gamma+\frac{1}{2X}, \label{eq:EMsum1} \\
\sum_{n\leq X} n^{2(\sigma-1)} \leq X^{2\sigma-1}\left(\log{X}+\gamma+\frac{1}{2X}\right) \label{eq:EMsum2}, \\
\sum_{n\leq X}n^{1-2\sigma} \leq \frac{X^{2(1-\sigma)}}{2(1-\sigma)}, \label{eq:Isum1} \\
\sum_{n\leq X}n^{-\sigma} \leq \frac{X^{1-\sigma}}{1-\sigma}, \label{eq:Isum2} \\
\sum_{n\leq X}n^{2\sigma-1} \leq \frac{X^{2\sigma}}{2\sigma}\left(1+\frac{2\sigma}{X}-\frac{1}{X^{2\sigma}}\right). \label{eq:Isum3}
\end{gather}
These bounds are good also for $\sigma=1/2$. Inequalities \eqref{eq:Isum1}, \eqref{eq:Isum2} and \eqref{eq:Isum3} follow simply from integration.

The next two lemmas are explicit versions of Selberg's Lemmas 2 and 3, with the same proof in principle. The first one is needed to estimate $B_2$ while the second one is useful to obtain bounds for $B_3$ and $B_4$.

\begin{lemma}
\label{lem:integral}
Let $\sigma\geq1/2$, $\lambda\neq 0$ and $T_1\leq T_2$. Then
\[
\left|\int_{T_1}^{T_2}t^{1-2\sigma}e^{\ie\lambda t}\dif{t}\right| \leq \frac{2}{|\lambda|}T_1^{1-2\sigma}.
\]
\end{lemma}

\begin{proof}
The stated inequality is clearly true in case of $\sigma=1/2$. If we assume that $\sigma>1/2$, it is not hard to see that integration by parts implies the stated bound.
\end{proof}

\begin{lemma}
\label{lem:integral2}
Let $\sigma\geq1/2$, $\xi\in\left(0,T_1\right]$ and $2\leq T_1<T_1+\sqrt{T_1}\leq T_2$. Then
\[
\left|\int_{T_1}^{T_2}t^{\frac{1}{2}-\sigma}\left(\frac{t}{e\xi}\right)^{\pm\ie t}\dif{t}\right| \leq \frac{8T_1^{\frac{1}{2}-\sigma}}{\log{\frac{T_1+\sqrt{T_1}}{\xi}}}.
\]
\end{lemma}

\begin{proof}
Denote by $I$ the above integral and assume $\xi\neq T_1$. Separating real and imaginary part of the exponential function, we obtain
\begin{flalign*}
I &= \int_{T_1}^{T_2}\frac{t^{\frac{1}{2}-\sigma}}{\log{\frac{t}{\xi}}}\cos{\left(t\log{\frac{t}{e\xi}}\right)\log{\frac{t}{\xi}}}\dif{t} \pm \ie\int_{T_1}^{T_2}\frac{t^{\frac{1}{2}-\sigma}}{\log{\frac{t}{\xi}}}\sin{\left(t\log{\frac{t}{e\xi}}\right)\log{\frac{t}{\xi}}}\dif{t} \\
&= \frac{T_1^{\frac{1}{2}-\sigma}}{\log{\frac{T_1}{\xi}}}\left(\int_{u\left(T_1\right)}^{u\left(T_1+\eta_1\left(T_2-T_1\right)\right)}\cos{u}\dif{u}\pm\ie \int_{u\left(T_1\right)}^{u\left(T_1+\eta_2\left(T_2-T_1\right)\right)}\sin{u}\dif{u}\right)
\end{flalign*}
for some $\eta_1,\eta_2\in[0,1]$. The second equality follows from the second mean value theorem and after making substitution $u(t)\de t\log{\left(t/\left(e\xi\right)\right)}$ with $u'(t)=\log{\left(t/\xi\right)}>0$. This implies that 
\begin{equation}
\label{eq:integral2}
|I| \leq \frac{4T_1^{\frac{1}{2}-\sigma}}{\log{\frac{T_1}{\xi}}}. 
\end{equation}
Define
\[
f\left(\xi\right)\de\frac{\log{\frac{T_1+\sqrt{T_1}}{\xi}}}{\log{\frac{T_1}{\xi}}}, \quad g\left(\xi\right)\de\frac{\sqrt{T_1}}{4}\log{\frac{T_1+\sqrt{T_1}}{\xi}}.
\]
The first function is strictly increasing while the second one is strictly decreasing. For $\xi\leq T_1-\sqrt{T_1}$ we thus have 
\[
\frac{1}{\log{\frac{T_1}{\xi}}} \leq \frac{f\left(T_1-\sqrt{T_1}\right)}{\log{\frac{T_1+\sqrt{T_1}}{\xi}}} \leq \frac{2}{\log{\frac{T_1+\sqrt{T_1}}{\xi}}}
\]
since $\lim_{T_1\to\infty}f\left(T_1-\sqrt{T_1}\right)=2$. In this case we obtain the desired inequality. Now, let $T_1-\sqrt{T_1}\leq\xi\leq T_1$. By the already known inequality \eqref{eq:integral2} we have
\begin{flalign*}
|I| &\leq \int_{T_1}^{T_1+\sqrt{T_1}}t^{\frac{1}{2}-\sigma}\dif{t} + \left|\int_{T_1+\sqrt{T_1}}^{T_2}t^{\frac{1}{2}-\sigma}\left(\frac{t}{e\xi}\right)^{\pm\ie t}\dif{t}\right| \\ 
&\leq T_1^{1-\sigma} + \frac{4T_1^{\frac{1}{2}-\sigma}}{\log{\frac{T_1+\sqrt{T_1}}{\xi}}}
\leq \frac{4T_1^{\frac{1}{2}-\sigma}}{\log{\frac{T_1+\sqrt{T_1}}{\xi}}}\left(1+g\left(T_1-\sqrt{T_1}\right)\right).
\end{flalign*}
This also proves the main bound since $g\left(T_1-\sqrt{T_1}\right)<1$. The proof of Lemma \ref{lem:integral2} is thus complete.
\end{proof}

We are now in position to obtain desired bounds for integrals $B_i$. We will do this in pairs of indices, namely for $\{5,6\}$, $\{7,8\}$, $\{1,2\}$ and $\{3,4\}$, but starting with the most simple one $B_9$. Derivation of bounds for one part of a pair give bounds for the other part when changing the roles of parameters $\mu_1$ and $\mu_2$. Note that the order of appearance of parameters $\mu_1$ and $\mu_2$ is crucial when obtaining bounds which depend only on $\sigma_0$ and $T_0$. Here we use inequalities $\mu_1/\mu_2\leq1$ and $\pi\mu_2/\left(T\mu_1\right)\leq1/2$. 

\subsection{Bound on $B_9$} 
\label{sec:B9}

A straightforward calculation shows that
\[
\left|B_9\right|\leq T^{1-\sigma}\sqrt{\frac{\mu_2}{\mu_1}}\mathcal{B}_9\left(\sigma_0\right),
\]
where 
\begin{equation}
\label{eq:b9}
\mathcal{B}_9\left(\sigma_0\right)\de \pi^{\sigma_0} \frac{2-2^{\sigma_0}}{1-\sigma_0}\widetilde{E}\cdot\widetilde{F}.
\end{equation}
Here we used the fact that $\left(2-2^x\right)/\left(1-x\right)$ is a strictly increasing function on $[1/2,1)$.

\subsection{Bounds on $B_5$ and $B_6$} 
\label{subsec:5and6}

Firstly, we will consider $B_5$. Define
\[
S_{x(t)}(s)\de \sum_{n\leq x(t)}\frac{1}{n^{\sigma+\ie t}}.
\]
H\"{o}lder's inequality implies
\[
\int_{T}^{2T} \left(\frac{t}{2\pi}\right)^{-\frac{\sigma}{2}}\left|S_{x(t)}(s)\right|\dif{t} \leq T^{\frac{1-\sigma}{2}}(2\pi)^{\frac{\sigma}{2}}\sqrt{\frac{2^{1-\sigma}-1}{1-\sigma}\mathcal{I}_{x(t)}(s)},
\]
where
\[
\mathcal{I}_{x(t)}(s) \de \int_{T}^{2T}\left|S_{x(t)}(s)\right|^2\dif{t}.
\]
The same inequality is also true for $y(t)$ and $\bar{s}=\sigma-\ie t$. The problem is thus reduced to bounding the second integral. Separation of the diagonal and off-diagonal terms which appear after multiplication gives, together with equality \eqref{eq:intsum},
\begin{equation}
\label{eq:diagonal}
\mathcal{I}_{x(t)}(s) = \int_{T}^{2T}S_{x(t)}(2\sigma)\dif{t} + \ie D\left(-\sigma,-M_1,-2T;x\left(2T\right)\right) 
\end{equation}
where $M_1\de\max\left\{T,x^{-1}(n),x^{-1}(m)\right\}$. Using \eqref{eq:EMsum1}, we can deduce by straightforward integration that
\[
\int_{T}^{2T}S_{x(t)}(2\sigma)\dif{t} \leq B_{5,1}\left(T,\mu_1,\mu_2\right)T\log{T},
\]
where
\[
B_{5,1}\left(T,\mu_1,\mu_2\right)\de \frac{1}{2}+\frac{\gamma-\frac{1}{2}+\log{\sqrt{\frac{2\mu_1}{\pi\mu_2}}}+\left(\sqrt{2}-1\right)\sqrt{\frac{2\pi\mu_2}{\mu_1 T}}}{\log{T}}.
\]
After changing roles of $\mu_1$ and $\mu_2$, the resulting bound is also true for $y(t)$ in place of $x(t)$ since \eqref{eq:diagonal} is also true in this case with $M_2\de\max\left\{T,y^{-1}(n),y^{-1}(m)\right\}$ in place of $M_1$. We shall see that $B_{5,1}$ contributes the most in \eqref{eq:diagonal}.
 
Using \eqref{eq:doublesumbound}, we get 
\[
\left|\ie D\left(-\sigma,-M_1,-2T;x\left(2T\right)\right)\right|\leq \left(\frac{\mu_1}{\pi\mu_2}\right)^{1-\sigma}B_{5,2}\left(z\right)T\log{T}
\] 
where
\begin{multline*}
B_{5,2}\left(z\right) \de \frac{T^{-\sigma}}{2(1-\sigma)} + \frac{4+2(1-\sigma)\left(2\gamma+\log{\frac{\mu_1}{\pi\mu_2}}+\sqrt{\frac{\pi\mu_2}{\mu_1T}}+\pi m_0\right)}{4(1-\sigma)^2T^{\sigma}\log{T}} \\
+\left(\frac{\pi\mu_2}{\mu_1 T}\right)^{1-\sigma}\frac{\pi m_0-2}{4T^{\sigma}}\left(1+\frac{2\gamma+\log{\frac{\mu_1}{\pi\mu_2}}+\sqrt{\frac{\pi\mu_2}{\mu_1T}}}{\log{T}}\right).
\end{multline*}
This bound is also true for $y(T)$ after changing roles of $\mu_1$ and $\mu_2$. Define
\begin{equation}
\label{eq:boundB5}
\begin{split}
\widetilde{B}_5\left(z\right) &\de B_{5,1}\left(T,\mu_1,\mu_2\right)+\left(\frac{\mu_1}{\pi\mu_2}\right)^{1-\sigma}B_{5,2}\left(z\right) \\
&\leq\!\begin{multlined}[t] \frac{1}{2}+\frac{0.266}{\log{T_0}}+\frac{\pi^{\sigma_0-1}}{2\left(1-\sigma_0\right)\sqrt{T_0}}\Biggl(1+\frac{4.43}{\left(1-\sigma_0\right)\log{T_0}} \\
+0.534\cdot 2^{\sigma_0-1}\left(1+\frac{0.717}{\log{T_0}}\right)\Biggr),
\end{multlined}
\end{split}
\end{equation}
\begin{flalign}
\widetilde{B}_6\left(z\right) &\de \left(\frac{\pi\mu_1}{\mu_2}\right)^{1-\sigma}B_{5,1}\left(T,\mu_2,\mu_1\right) + B_{5,2}\left(\sigma,T,\mu_2,\mu_1\right) \nonumber \\
&\leq \sqrt{\pi} + \frac{1}{\left(1-\sigma_0\right)\sqrt{T_0}}\left(1+\frac{1.577+\frac{1}{4}\sqrt{\frac{\pi}{T_0}}}{\left(1-\sigma_0\right)\log{T_0}}+0.534\left(\frac{\pi}{T_0}\right)^{1-\sigma_0}\right). \label{eq:boundB6}
\end{flalign}
In derivation of the second inequality we used $B_{5,1}\left(T,\mu_2,\mu_1\right)<1$. Both functions and their bounds are continuous for $\sigma\in\left[1/2,\sigma_0\right]$ and $\sigma_0\in[1/2,1)$. Observe also that $B_{5,1}(T,1,1)<1/2$ for $T\geq50$. This gives
\[
\widetilde{B}_5\left(\frac{1}{2},T,1,1\right)\leq \frac{1}{2}+2.3\sqrt{\frac{1}{\pi T_0}}, \quad 
\widetilde{B}_6\left(\frac{1}{2},T,1,1\right)\leq \frac{\sqrt{\pi}}{2}+2.3\sqrt{\frac{1}{T_0}}
\]
for $T_0\geq50$. In case $\mu_1=\mu_2=1$ and $\sigma=1/2$ we will use these bounds instead of \eqref{eq:boundB5} and \eqref{eq:boundB6}. We have
\[
\left|B_5\right| \leq \mathcal{B}_5(z)\left(\frac{\mu_2}{\mu_1}\right)^{\frac{1-\sigma}{2}}T^{1-\frac{\sigma}{2}}\sqrt{\log{T}}, \quad \left|B_6\right| \leq \mathcal{B}_6(z)\left(\frac{\mu_2}{\mu_1}\right)^{\frac{1}{2}}T^{1-\frac{\sigma}{2}}\sqrt{\log{T}},
\]
where
\begin{gather}
    \mathcal{B}_5\left(z\right) \de \widetilde{F}\pi^{\frac{\sigma_0}{2}} \sqrt{\frac{2-2^{\sigma_0}}{1-\sigma_0}\widetilde{B}_5\left(z\right)}, \label{eq:b5} \\
    \mathcal{B}_6\left(z\right) \de \widetilde{E}\pi^{\sigma_0-\frac{1}{2}}\sqrt{ \frac{2-2^{\sigma_0}}{1-\sigma_0}\widetilde{B}_6\left(z\right)}. \label{eq:b6}
\end{gather}
In the general case we will use \eqref{eq:boundB5} and \eqref{eq:boundB6} to bound $\mathcal{B}_5\left(z\right)$ and $\mathcal{B}_6\left(z\right)$. This implies that both functions are bounded for fixed $\sigma_0$ and $T_0$.

\subsection{Bounds on $B_7$ and $B_8$} 
\label{sec:B7B8}

The strategy here is the same as in Section \ref{subsec:5and6}. H\"{o}lder's inequality implies
\[
\int_{T}^{2T} \left(\frac{t}{2\pi}\right)^{\frac{1-3\sigma}{2}}\left|S_{y(t)}(1-s)\right|\dif{t} \leq T^{\frac{2-3\sigma}{2}}(2\pi)^{\frac{3\sigma-1}{2}}\sqrt{\frac{2^{2-3\sigma}-1}{2-3\sigma}\mathcal{I}_{y(t)}(1-s)},
\]
and we have
\begin{equation}
\label{eq:diagonal2}
\mathcal{I}_{y(t)}(1-s) = \int_{T}^{2T}S_{y(t)}(2(1-\sigma))\dif{t} + \ie D\left(\sigma-1,-M_2,-2T;y(2T)\right). 
\end{equation}
Using inequality \eqref{eq:EMsum2}, we can estimate by straightforward integration that 
\[
\int_{T}^{2T}S_{y(t)}(2(1-\sigma))\dif{t}\leq B_{7,1}\left(z\right)T^{\frac{1}{2}+\sigma}\left(\frac{\mu_2}{\pi\mu_1}\right)^{\sigma}\log{\frac{T\mu_2}{\pi\mu_1}},
\]
where 
\begin{multline*}
B_{7,1}\left(z\right) \de (1+2\sigma)2^{\frac{1}{2}+\sigma}-\frac{2^{\frac{1}{2}-\sigma}}{1+2\sigma}+\frac{1}{\log{\frac{T\mu_2}{\pi\mu_1}}}\left(-\frac{2^{\frac{1}{2}-\sigma}\left(2^{\frac{3}{2}+\sigma}-2\right)}{(1+2\sigma)^2}\right. \\
\left.+\frac{\left(4-2^{\frac{3}{2}-\sigma}\right)\gamma+2^{\frac{1}{2}-\sigma}\log{2}}{1+2\sigma}+\frac{1-2^{-\sigma}}{\sigma}\sqrt{\frac{\pi\mu_1}{T\mu_2}}\right).
\end{multline*}
By inequality \eqref{eq:doublesumbound} we also have
\[
\left|\ie D\left(\sigma-1,-M_2,-2T;y(2T)\right)\right| \leq B_{7,2}\left(z\right)\left(\frac{\mu_2}{\pi\mu_1}\right)^{\sigma}T^{\frac{1}{2}+\sigma}\log{\frac{T\mu_2}{\pi\mu_1}},
\]
where
\begin{multline*}
B_{7,2}\left(z\right)\de \frac{1+2\sigma\sqrt{\frac{\pi\mu_1}{T\mu_2}}-\left(\frac{\pi\mu_1}{T\mu_2}\right)^\sigma}{2\sigma\sqrt{T}}\left(1+\frac{\frac{2}{\sigma}+\pi m_0+2\gamma+\sqrt{\frac{\pi\mu_1}{T\mu_2}}}{\log{\frac{T\mu_2}{\pi\mu_1}}}\right) \\ 
+ \frac{\pi m_0-2}{4\sqrt{T}}\sqrt{\frac{\pi\mu_1}{\mu_2 T}}\left(1+\frac{2\gamma+\sqrt{\frac{\pi\mu_1}{T\mu_2}}}{\log{\frac{T\mu_2}{\pi\mu_1}}}\right).
\end{multline*}
After changing roles of $\mu_1$ and $\mu_2$, both bounds are also true for $x(t)$ and $M_1$ in place of $y(t)$ and $M_2$, respectively. Define
\begin{flalign}
\widetilde{B}_7\left(z\right) &\de \sqrt{\frac{\pi\mu_1}{\mu_2}}B_{7,1}\left(z\right)+B_{7,2}\left(z\right) \nonumber \\
&\leq \sqrt{\pi}\Biggl(\left(1+2\sigma_0\right)2^{\frac{1}{2}+\sigma_0}-\frac{2^{\frac{1}{2}-\sigma_0}}{1+2\sigma_0} \label{eq:boundB7} \\
&+\frac{1}{\log{\frac{T_0}{\pi}}}\left(\frac{\left(4-2^{\frac{3}{2}-\sigma_0}\right)\gamma+\log{2}}{2}
-\frac{2^{\frac{3}{2}-\sigma_0}}{\left(1+2\sigma_0\right)^2}+2\left(1-2^{-\sigma_0}\right)\sqrt{\frac{\pi}{T_0}}\right)\Biggr) \nonumber \\
&+\frac{1+2\sigma_0\sqrt{\frac{\pi}{T_0}}}{\sqrt{T_0}}\left(1+\frac{9.287+\sqrt{\frac{\pi}{T_0}}}{\log{\frac{T_0}{\pi}}}\right)
+\frac{0.534\sqrt{\pi}}{T_0}\left(1+\frac{2\gamma+\sqrt{\frac{\pi}{T_0}}}{\log{\frac{T_0}{\pi}}}\right), \nonumber
\end{flalign}

\begin{flalign}
\widetilde{B}_8\left(z\right) &\de B_{7,1}\left(\sigma,T,\mu_2,\mu_1\right)+\sqrt{\frac{\mu_1}{\pi\mu_2}}B_{7,2}\left(\sigma,T,\mu_2,\mu_1\right) \nonumber \\
&\leq \left(1+2\sigma_0\right)2^{\frac{1}{2}+\sigma_0}-\frac{2^{\frac{1}{2}-\sigma_0}}{1+2\sigma_0} \label{eq:boundB8} \\
&+\frac{1}{\log{2}}\left(\frac{\left(4-2^{\frac{3}{2}-\sigma_0}\right)\gamma+\log{2}}{2}
-\frac{2^{\frac{3}{2}-\sigma_0}}{\left(1+2\sigma_0\right)^2}+\left(1-2^{-\sigma_0}\right)\sqrt{2}\right)+\frac{27.7101}{\sqrt{\pi T_0}}. \nonumber
\end{flalign}
Both functions and their bounds are continuous for $\sigma\in\left[1/2,\sigma_0\right]$ and $\sigma_0\in[1/2,1)$. In case $\mu_1=\mu_2=1$ and $\sigma=1/2$ we will use 
\begin{gather*}
\widetilde{B}_7\left(\frac{1}{2},T,1,1\right) \leq \sqrt{\pi}B_{7,1}\left(\frac{1}{2},T_0,1,1\right)+B_{7,2}\left(\frac{1}{2},T_0,1,1\right), \\
\widetilde{B}_8\left(\frac{1}{2},T,1,1\right) \leq B_{7,1}\left(\frac{1}{2},T_0,1,1\right)+\sqrt{\frac{1}{\pi}}B_{7,2}\left(\frac{1}{2},T_0,1,1\right)
\end{gather*}
instead of \eqref{eq:boundB7} and \eqref{eq:boundB8}. Then
\[
\left|B_7\right| \leq \mathcal{B}_7(z) \left(\frac{\mu_2}{\mu_1}\right)^{\frac{1}{2}}T^{\frac{5}{4}-\sigma}\sqrt{\log{\frac{T\mu_2}{\pi\mu_1}}}, \quad \left|B_8\right| \leq \mathcal{B}_8(z) \left(\frac{\mu_2}{\mu_1}\right)^{\frac{2\sigma+1}{4}}T^{\frac{5}{4}-\sigma}\sqrt{\log{\frac{T\mu_1}{\pi\mu_2}}},
\]
where
\begin{gather}
\mathcal{B}_7\left(z\right) \de \widetilde{F}\cdot\pi^{\sigma_0-\frac{1}{2}}\sqrt{\lambda_1\left(\sigma_0\right)\widetilde{B}_7\left(z\right)}, \label{eq:b7} \\
\mathcal{B}_8\left(z\right) \de \widetilde{E}\cdot\pi^{\sigma_0-\frac{1}{4}}\sqrt{\lambda_1\left(\sigma_0\right)\widetilde{B}_8\left(z\right)}, \label{eq:b8}
\end{gather}
and $\lambda_1(x)$ is a continuous function on $\R$, defined as
\begin{equation*}
    \lambda_1\left(x\right) \de 
    \left\{
    \begin{array}{ll}
       \frac{2-2^{3x-1}}{2-3x}, & x\neq 2/3, \\
        2\log{2}, & x=2/3. 
    \end{array}
    \right.
\end{equation*}
Observe that $\lambda_1(x)$ is a strictly increasing function on $[1/2,1]$. In the general case we use \eqref{eq:boundB7} and \eqref{eq:boundB8} to bound $\mathcal{B}_7\left(z\right)$ and $\mathcal{B}_8\left(z\right)$, which means that both functions are bounded for fixed $\sigma_0$ and $T_0$.

\subsection{Bounds on $B_1$ and $B_2$.}
\label{sec:bounds12}

Separation of the diagonal (here we need coprimality of $\mu_1$ and $\mu_2$) and off-diagonal terms in $A_1$ and $A_2$, together with equality \eqref{eq:intsum} gives
\begin{flalign*}
B_1 &= \left(\mu_1\mu_2\right)^{-\sigma}\int_T^{2T}\sum_{n\leq x\left(t/\mu_1^2\right)}n^{-2\sigma}\dif{t}+ \mathop{\sum_{n\leq x(2T)}\sum_{m\leq y(2T)}}_{m\mu_1\neq n\mu_2}\int_{M_3}^{2T}(nm)^{-\sigma}\left(\frac{m\mu_1}{n\mu_2}\right)^{\ie t}\dif{t}. 
\end{flalign*}
where $M_3\de \max\left\{T,x^{-1}(n),y^{-1}(m)\right\}$, and
\begin{multline*}
B_2 = \left(\mu_1\mu_2\right)^{\sigma-1}\int_T^{2T}\left(\frac{2\pi}{t}\right)^{2\sigma-1}\sum_{n\leq y\left(t/\mu_2^2\right)}n^{2\sigma-2}\dif{t} \\
+ \mathop{\sum_{n\leq y(2T)}\sum_{m\leq x(2T)}}_{n\mu_1\neq m\mu_2}\int_{M_4}^{2T}\left(\frac{2\pi}{t}\right)^{2\sigma-1}(nm)^{\sigma-1}\left(\frac{n\mu_1}{m\mu_2}\right)^{\ie t}\dif{t}.
\end{multline*}
where $M_4\de \max\left\{T,y^{-1}(n),x^{-1}(m)\right\}$. Denote by $B_{1,1}$ and $B_{1,2}$ the first integral and the double sum in $B_1$, respectively. In the same vein define also $B_{2,1}$ and $B_{2,2}$. If we apply Theorem \ref{thm:fafe} on the sums in $B_{1,1}$ and $B_{2,1}$, and knowing that $\left|R(\sigma;x)\right|\leq (1/2)x^{-\sigma}$, then we obtain
\begin{equation*}
B_{1,1} + B_{2,1} = \mathscr{S}(z) - \frac{(2\pi)^{\sigma}\left(2^{1-\sigma}-1\right)}{1-\sigma}T^{1-\sigma},
\end{equation*}
where $\mathscr{S}(z)$ is defined by \eqref{eq:selbergS}. Writing $m\mu_1=M$ and $n\mu_2=N$, we get
\[
\left|B_{1,2}\right| \leq 2\left(\mu_1\mu_2\right)^\sigma \mathop{\sum\sum}_{\substack{N,M\leq \sqrt{\frac{T\mu_1\mu_2}{\pi}}\\ N\neq M}}\frac{(NM)^{-\sigma}}{\left|\log{(N/M)}\right|}
\leq \mu_1\mu_2\widetilde{B}_{1,2}\left(\sigma_0,T_0\right) T^{1-\sigma}\log{\frac{T\mu_1\mu_2}{\pi}},
\]
where
\begin{equation*}
\widetilde{B}_{1,2}\left(\sigma_0,T_0\right) \de \frac{\pi^{\sigma_0-1}}{1-\sigma_0}\left(1+\frac{2\gamma+\frac{2}{1-\sigma_0}+\sqrt{\frac{\pi}{T_0}}}{\log{\frac{T_0}{\pi}}}\right)
\end{equation*}
Lemma \ref{lem:integral} implies
\begin{flalign*}
\left|B_{2,2}\right| &\leq 2(2\pi)^{2\sigma-1}\left(\mu_1\mu_2\right)^{1-\sigma}T^{1-2\sigma} \mathop{\sum\sum}_{\substack{N,M\leq \sqrt{\frac{T\mu_1\mu_2}{\pi}}\\ N\neq M}}\frac{(NM)^{\sigma-1}}{\left|\log{(N/M)}\right|}\\
&\leq \mu_1\mu_2\widetilde{B}_{2,2}\left(\sigma_0,T_0\right)T^{1-\sigma}\log{\frac{T\mu_1\mu_2}{\pi}},
\end{flalign*}
where 
\[
\widetilde{B}_{2,2}\left(\sigma_0,T_0\right)\de 2^{2\sigma_0}\pi^{\sigma_0-1}\left(1+2\sigma_0\sqrt{\frac{\pi}{T_0}}+\frac{4+2\gamma+\left(1+4\gamma\sigma_0\right)\sqrt{\frac{\pi}{T_0}}+\frac{2\pi\sigma_0}{T_0}}{\log{\frac{T_0}{\pi}}}\right).
\]
In both cases we have used the inequality from Lemma \ref{lem:doublesum} without the term with the minus sign. We get 
\[
\left|B_1+B_2-\mathscr{S}(z)\right| \leq \mathcal{B}_1\left(\sigma_0,T_0\right)\mu_1\mu_2T^{1-\sigma}\log{\frac{T\mu_1\mu_2}{\pi}},
\]
where
\begin{equation}
\label{eq:b1}
    \mathcal{B}_1\left(\sigma_0,T_0\right) \de \widetilde{B}_{1,2}\left(\sigma_0,T_0\right)+\widetilde{B}_{2,2}\left(\sigma_0,T_0\right)+\frac{\pi^{\sigma_0}\left(2-2^{\sigma_0}\right)}{\left(1-\sigma_0\right)\log{\frac{T_0}{\pi}}}.
\end{equation}
This function is clearly bounded for fixed $\sigma_0$ and $T_0$, and is also continuous for $\sigma_0\in[1/2,1)$.

\subsection{Bounds on $B_3$ and $B_4$}
\label{sub:B3B4}

Using \eqref{eq:tildechi} and \eqref{eq:intsum}, we obtain
\[
B_{3} = e^{\frac{\pi}{4}\ie}(2\pi)^{\sigma-\frac{1}{2}}\mathop{\sum\sum}_{n,m\leq y(2T)}n^{\sigma-1}m^{-\sigma}\int_{M_2(n,m)}^{2T}t^{\frac{1}{2}-\sigma}\left(\frac{t}{2\pi enm\frac{\mu_1}{\mu_2}}\right)^{-\ie t}\dif{t}.
\]
The equation for $B_4$ is the same except that we need to replace $e^{\ie\pi/4}$ by $e^{-\ie\pi/4}$, $y(2T)$ by $x(2T)$, and $-\ie t$ by $\ie t$. Lemma \ref{lem:integral2} and separation of diagonal and off-diagonal terms imply 
\[
\left|B_3\right| \leq 8(2\pi)^{\sigma-\frac{1}{2}}T^{\frac{1}{2}-\sigma}\left(B_{3,1}+B_{3,2}\right),
\]
where 
\begin{gather*}
B_{3,1} \de \sum_{n\leq y(2T)}\left(n\log{\frac{M_2(n,n)+\sqrt{M_2(n,n)}}{2\pi n^2\frac{\mu_1}{\mu_2}}}\right)^{-1}, \\
B_{3,2} \de \mathop{\sum\sum}_{\substack{n,m\leq y(2T)\\ n\neq m}} n^{2\sigma-1}\left((nm)^{\sigma}\log{\frac{M_2(n,m)+\sqrt{M_2(n,m)}}{2\pi nm\frac{\mu_1}{\mu_2}}}\right)^{-1}.
\end{gather*}
The same inequality holds also for $B_4$ except that the summation goes up to $x(2T)$. Using the fact that $M_2(n,n)\geq 2\pi n^2\mu_1/\mu_2$, $\sqrt{x}\log{\left(1+1/\sqrt{x}\right)}\geq \sqrt{T}/\left(1+\sqrt{T}\right)$ for $x\geq T$, and $\sqrt{M_2(n,n)}\leq\sqrt{2T}$, we get
\[
B_{3,1} \leq \widetilde{B}_{3,1}\left(T,\mu_1,\mu_2\right)\sqrt{\frac{\mu_2}{\pi\mu_1}}\sqrt{T}\log{\frac{T\mu_2}{\pi\mu_1}}, 
\]
where
\begin{flalign}
\widetilde{B}_{3,1}\left(T,\mu_1,\mu_2\right) &\de \sqrt{\frac{\pi\mu_1}{2\mu_2}}\left(1+\frac{2\gamma+\sqrt{\frac{\pi\mu_1}{T\mu_2}}}{\log{\frac{T\mu_2}{\pi\mu_1}}}\left(1+\frac{1}{\sqrt{T}}\right)+\frac{1}{\sqrt{T}}\right) \nonumber \\
&\leq \sqrt{\frac{\pi}{2}}\left(1+\frac{2\gamma+\sqrt{\frac{\pi}{T_0}}}{\log{\frac{T_0}{\pi}}}\left(1+\frac{1}{\sqrt{T_0}}\right)+\frac{1}{\sqrt{T_0}}\right). \label{eq:boundB31}
\end{flalign}
Because
\[
\frac{M_2(n,m)+\sqrt{M_2(n,m)}}{2\pi nm\frac{\mu_1}{\mu_2}}\geq \frac{M_2(n,m)}{2\pi nm\frac{\mu_1}{\mu_2}}=\max{\left\{\frac{T\mu_2}{2\pi nm\mu_1},\frac{n}{m},\frac{m}{n}\right\}},
\]
it follows by Lemma \ref{lem:doublesum} that
\[
B_{3,2} \leq \left(\frac{T\mu_2}{\pi\mu_1}\right)^{\sigma-\frac{1}{2}}\mathop{\sum\sum}_{\substack{n,m\leq y(2T)\\ n\neq m}}\frac{(nm)^{-\sigma}}{\left|\log{(n/m)}\right|} \leq \widetilde{B}_{3,2}(z)\sqrt{\frac{\mu_2}{\pi\mu_1}}\sqrt{T}\log{\frac{T\mu_2}{\pi\mu_1}},
\]
where
\begin{flalign}
\widetilde{B}_{3,2}\left(z\right) &\de \frac{1}{2(1-\sigma)}\left(1+\frac{1}{\log{\frac{T\mu_2}{\pi\mu_1}}}\left(\frac{2}{1-\sigma}+\frac{2\gamma}{\sqrt{\frac{T\mu_2}{\pi\mu_1}}}+\frac{\pi\mu_1}{T\mu_2}\right)\right) \nonumber \\
&\leq \frac{1}{2\left(1-\sigma_0\right)}\left(1+\frac{1}{\log{\frac{T_0}{\pi}}}\left(\frac{2}{1-\sigma_0}+2\gamma\sqrt{\frac{\pi}{T_0}}+\frac{\pi}{T_0}\right)\right). \label{eq:boundB32}
\end{flalign}
Define also
\begin{flalign}
\widetilde{B}_{4,1}\left(T,\mu_1,\mu_2\right) &\de \frac{\sqrt{2}}{2}+\frac{2\gamma+\sqrt{\frac{\pi\mu_2}{T\mu_1}}}{\sqrt{2}\log{\frac{T\mu_1}{\pi\mu_2}}}\left(1+\frac{1}{\sqrt{T}}\right)+\frac{1}{\sqrt{2T}} \nonumber \\ 
&\leq 2.607\left(1+\frac{1}{\sqrt{T_0}}\right), \label{eq:boundB41}
\end{flalign}
\begin{equation}
\label{eq:boundB42}
\widetilde{B}_{4,2}\left(z\right) \de \sqrt{\frac{\mu_1}{\pi\mu_2}} \widetilde{B}_{3,2}\left(\sigma,T,\mu_2,\mu_1\right) \leq \frac{1}{2\sqrt{\pi}\left(1-\sigma_0\right)}\left(2.9+\frac{2.886}{1-\sigma_0}\right). 
\end{equation}
Functions $\widetilde{B}_{3,2}(z)$ and $\widetilde{B}_{4,2}(z)$, and their bounds are continuous for $\sigma\in\left[1/2,\sigma_0\right]$ and $\sigma_0\in[1/2,1)$. In case $\mu_1=\mu_2=1$ and $\sigma=1/2$ it is better to use
\[
\widetilde{B}_{4,1}\left(T,1,1\right)\leq \widetilde{B}_{4,1}\left(T_0,1,1\right), \quad \widetilde{B}_{4,2}\left(\frac{1}{2},T,1,1\right)\leq \sqrt{\frac{1}{\pi}}\widetilde{B}_{3,2}\left(\frac{1}{2},T_0,1,1\right).
\]
Putting all together finally gives
\[
\left|B_3\right| \leq \mathcal{B}_3(z)\sqrt{\frac{\mu_2}{\pi\mu_1}}T^{1-\sigma}\log{\frac{T\mu_2}{\pi\mu_1}}, \quad
\left|B_4\right| \leq \mathcal{B}_4(z) T^{1-\sigma}\log{\frac{T\mu_1}{\pi\mu_2}},
\]
where
\begin{gather}
\mathcal{B}_3(z) \de 8(2\pi)^{\sigma_0-\frac{1}{2}}\left(\widetilde{B}_{3,1}\left(T,\mu_1,\mu_2\right)+\widetilde{B}_{3,2}\left(z\right)\right), \label{eq:b3} \\
\mathcal{B}_4(z) \de 8(2\pi)^{\sigma_0-\frac{1}{2}}\left(\widetilde{B}_{4,1}\left(T,\mu_1,\mu_2\right)+\widetilde{B}_{4,2}\left(z\right)\right). \label{eq:b4}
\end{gather}
In the general case we will use \eqref{eq:boundB31}, \eqref{eq:boundB32}, \eqref{eq:boundB41} and \eqref{eq:boundB42} to bound $\mathcal{B}_3(z)$ and $\mathcal{B}_4(z)$. This also implies that both functions are bounded for fixed $\sigma_0$ and $T_0$.

\subsection{Proof of Theorem \ref{thm:SelbergMoment}} The statement of Theorem \ref{thm:SelbergMoment} now easily follows by using bounds for $B_i$ developed in Sections \ref{subsec:5and6} and \ref{sec:B7B8}, which give $\mathscr{S}_1(z)$, and Sections \ref{sec:bounds12}, \ref{sub:B3B4} and \ref{sec:B9}, which give $\mathscr{S}_2(z)$.

\section{Explicit Selberg's zero density result}
\label{sec:eszd}

\subsection{The mollifier.} Let $s=\sigma+\ie t$ with $\sigma\geq1/2$ and $X\geq 1$. Selberg introduced $S_X(s)\de \sum_{n\leq X}\lambda_X(n)n^{-s}$ where
\[
\lambda_X(n)\de n^{2\sigma}\left(\sum_{m\leq X}\frac{\mu^2(m)}{\varphi_{2\sigma}(m)}\right)^{-1}\sum_{m\leq X/n}\frac{\mu(nm)\mu(m)}{\varphi_{2\sigma}(nm)}
\]
and 
\[
\varphi_{x}(n)\de n^{x}\sum_{d|n}\frac{\mu(d)}{d^{x}} = n^{x}\prod_{p|n}\left(1-p^{-x}\right)
\]
for $x\in\R$. Observe that $\varphi_1(n)$ is the ordinary Euler totient function $\varphi(n)$, and also that $\lambda_X(1)=1$. Because $\mu(n)$ and $\varphi_{2\sigma}(n)$ are multiplicative functions, and $\mu(nm)=0$ if $(m,n)\neq 1$, it follows 
\[
\lambda_X(n) = \frac{\mu(n)n^{2\sigma}}{\varphi_{2\sigma}(n)}\left(\sum_{m\leq X}\frac{\mu^2(m)}{\varphi_{2\sigma}(m)}\right)^{-1}\sum_{\substack{m\leq X/n\\ (m,n)=1}}\frac{\mu^2(m)}{\varphi_{2\sigma}(m)}.
\]
This implies
\[
\left|\lambda_X(n)\right| \leq \frac{n^{2\sigma}}{\varphi_{2\sigma}(n)}=\prod_{p|n}\frac{1}{1-p^{-2\sigma}}.
\]
If $\sigma>1/2$, then $\left|\lambda_X(n)\right| \leq \zeta(2\sigma)$. Let $1<n\leq X$. Then the above product is not greater than the same product for $p\leq X$ and $\sigma=1/2$. Therefore,
\begin{equation}
\label{eq:lambdaGB}
\left|\lambda_X(n)\right| \leq \widehat{\lambda}(X)\log{X} 
\end{equation}
for some bounded function $\widehat{\lambda}(X)\leq 2.2$ where $X\geq8$, see \cite[Corollary 1]{RosserSchoenfeld}.

\begin{lemma}
\label{lem:sellem1}
We have
\begin{gather*}
\mathop{\sum\sum}_{n,m\leq X} \frac{\lambda_X(n)\lambda_X(m)}{(nm)^{2\sigma}}(n,m)^{2\sigma}=\left(\sum_{k\leq X}\frac{\mu^2(k)}{\varphi_{2\sigma}(k)}\right)^{-1}, \\
\mathop{\sum\sum}_{n,m\leq X} \frac{\lambda_X(n)\lambda_X(m)}{nm}(n,m)^{2-2\sigma}>0.
\end{gather*}
\end{lemma}

\begin{proof}
This is quite straightforward to prove if we notice that 
\[
\left(n,m\right)^{x}=\sum_{d|n,d|m}\varphi_{x}(d). 
\]
For details see \cite[pp.~15--16]{SelbergContrib} or \cite[p.~401]{KaratKor}.
\end{proof}

\begin{lemma}
\label{lem:sellem2}
Let $\sigma\geq 1/2+1/\log{X}$. Then 
\[
\zeta(2\sigma)\mathop{\sum\sum}_{n,m\leq X} \frac{\lambda_X(n)\lambda_X(m)}{(nm)^{2\sigma}}(n,m)^{2\sigma}\leq 1 + \frac{1+\frac{2\sigma-1}{X}}{1-e^{-2}}X^{1-2\sigma}.
\]
\end{lemma}

\begin{proof}
By Lemma \ref{lem:sellem1} we have
\begin{flalign*}
\zeta(2\sigma)\mathop{\sum\sum}_{n,m\leq X} \frac{\lambda_X(n)\lambda_X(m)}{(nm)^{2\sigma}}(n,m)^{2\sigma} &= \zeta(2\sigma)\left(\sum_{k\leq X}\frac{\mu^2(k)}{\varphi_{2\sigma}(k)}\right)^{-1} \\ 
&\leq \frac{\zeta(2\sigma)}{\sum_{k\leq X}k^{-2\sigma}} = 1 + \frac{\sum_{k>X}k^{-2\sigma}}{\sum_{k\leq X}k^{-2\sigma}}.
\end{flalign*}
Because
\[
\sum_{k>X}k^{-2\sigma} \leq \frac{X^{1-2\sigma}}{2\sigma-1}\left(1+\frac{2\sigma-1}{X}\right), \quad \sum_{k\leq X}k^{-2\sigma} \geq \frac{1-e^{-2}}{2\sigma-1},
\]
the stated inequality follows.
\end{proof}

\subsection{Littlewood's lemma}

Let $f(s)$ be a holomorphic function on some domain in the complex plane which includes a rectangle with vertices $\sigma+\ie T$, $a+\ie T$, $a+\ie 2T$ and $\sigma+\ie 2T$, where $1/2\leq\sigma<1<a$. Denote by $n_f\left(\tau\right)$ the number of zeros of $f(s)$ in the set $\left\{s\in\C\colon \sigma<\tau<\Re\{s\}<a,T<\Im\{s\}<2T\right\}$, and assume that no zeros are on the boundary of the rectangle. Then Littlewood's lemma asserts that
\begin{multline*}
2\pi \int_{\sigma}^{a} n_f\left(\tau\right)\dif{\tau} = \int_{T}^{2T}\log{\left|f(\sigma+\ie t)\right|}-\log{\left|f(a+\ie t)\right|}\dif{t} \\ 
+ \int_{\sigma}^{a}\arg{f(\tau+ \ie 2T)}-\arg{f(\tau+ \ie T)}\dif{\tau}.
\end{multline*}
Define $\Phi_X(s)\de\zeta(s)S_X(s)$. Then $\Phi_X(s)$ is holomorphic in $\C\setminus\{1\}$ and $N(\tau,2T)-N(\tau,T)=n_{\zeta}(\tau)\leq n_{\Phi}(\tau)$. We need some trivial estimates on $\Phi_X(s)$ in order to apply Littlewood's lemma.

\begin{lemma}
\label{lem:PhiArg}
Let $\sigma\geq\sigma_1\geq2.4$ and define 
\begin{equation}
\label{eq:h}
h(\sigma)\de \left(2\zeta(\sigma)+\left(\frac{3}{2}\right)^{-\sigma}-1\right)\left(1+\zeta(\sigma)\zeta(2\sigma)\right).
\end{equation}
Then $\left|\Phi_X(s)-1\right|\leq 2^{-\sigma}h\left(\sigma_1\right)$ and 
\[
\left|\arg{\Phi_X(s)}\right|\leq 2^{-\sigma}\frac{h\left(\sigma_1\right)}{1-2^{-\sigma_1}h\left(\sigma_1\right)}.
\]
\end{lemma}

\begin{proof}
By definition of $\Phi_X(s)$, we have for $\sigma>1$ the following:
\[
\Phi_X(s) = 1 + \sum_{n=2}^{\infty} n^{-s} + \sum_{n=2}^{\fl{X}}\frac{\lambda_{X}(n)}{n^{s}}+\sum_{n=2}^{\infty}\sum_{m=2}^{\fl{X}}\frac{\lambda_X(m)}{(nm)^{s}}.
\]
Then the first inequality easily follows since $\left|\lambda_X(n)\right|\leq\zeta(2\sigma)$, 
\[
\left|\sum_{n=2}^{\infty} n^{-s}\right|\leq \zeta(\sigma)-1\leq 2^{-\sigma}\left(1+\left(\frac{3}{2}\right)^{-\sigma}+\sum_{n=4}^{\infty}\left(\frac{n}{2}\right)^{-\sigma}\right) \leq \frac{2^{-\sigma}h\left(\sigma\right)}{1+\zeta(\sigma)\zeta(2\sigma)}, 
\]
and $h(\sigma)$ is decreasing function. Similarly, we also have $\left|\Im\left\{\Phi_X(s)\right\}\right|\leq 2^{-\sigma}h\left(\sigma_1\right)$ and $\Re\left\{\Phi_X(s)\right\}\geq 1-2^{-\sigma_1}h\left(\sigma_1\right)>0$. This implies the bound on the argument of $\Phi_X(s)$.
\end{proof}

Lemma \ref{lem:PhiArg} assures that $\lim_{a\to\infty}\Phi_X(a+\ie t)=1$, and also that $n_{\Phi}(\tau)=0$ for $\tau\geq1$. Therefore, Littlewood's lemma implies 
\begin{flalign*}
\int_{\sigma}^{1}N(\tau,2T)-N(\tau,T)\dif{\tau} &\leq \frac{1}{2\pi}\int_{T}^{2T}\log{\left|\Phi_X(\sigma+\ie t)\right|}\dif{t} \\
&+ \frac{1}{2\pi}\int_{\sigma}^{\sigma_1}\left|\arg{\Phi_X(\tau+ \ie 2T)}\right|+\left|\arg{\Phi_X(\tau+ \ie T)}\right|\dif{\tau} \\
&+ \frac{2^{-\sigma_1}h\left(\sigma_1\right)}{\left(1-2^{-\sigma_1}h\left(\sigma_1\right)\right)\pi\log{2}}
\end{flalign*}
for $\sigma_1\geq2.4$. In the following two subsections we will provide explicit estimates of the above integrals. 

\subsection{Explicit upper bound for $\int_{T}^{2T}\log{\left|\Phi_X(\sigma+\ie t)\right|}\dif{t}$.}

We will use Theorem \ref{thm:SelbergMoment} together with Lemmas \ref{lem:sellem1} and \ref{lem:sellem2} in order to estimate 
\[
\int_{T}^{2T}\left|\Phi_X(\sigma+\ie t)\right|^2\dif{t} = \mathop{\sum\sum}_{n,m\leq X} \frac{\lambda_X(n)\lambda_X(m)S\left(\sigma,T,n,m\right)}{(nm)^\sigma}.
\]
The following proposition is an explicit version of \cite[Lemma 7]{SelbergContrib} for $H=T$. It is important to note that in this approach it is crucial to keep $\sigma$ far enough from $1/2$.   

\begin{proposition}
\label{prop:Phi}
Let $\sigma_0\in(1/2,1)$, $T\geq T_0>e^{\frac{6}{\sigma_0-1/2}}$, $X=T^{\frac{1}{6}-\varepsilon}$, and 
\begin{equation}
\label{eq:epsilon}
0 < \varepsilon \leq \frac{1}{6} - \frac{1}{\left(\sigma_0-\frac{1}{2}\right)\log{T_0}}.
\end{equation}
Then
\[
\int_{T}^{2T}\left|\Phi_X(\sigma+\ie t)\right|^2\dif{t} \leq T +\phi\left(\sigma_0,T_0,\varepsilon,T\right)T^{1-2\left(\frac{1}{6}-\varepsilon\right)\left(\sigma-\frac{1}{2}\right)}
\]
for 
\begin{equation}
\label{eq:sigma}
\sigma\in\left[\frac{1}{2}+\frac{1}{\left(\frac{1}{6}-\varepsilon\right)\log{T}},\sigma_0\right],
\end{equation}
where
\[
\phi\left(\sigma_0,T_0,\varepsilon,T\right) \de \frac{1+\left(2\sigma_0-1\right)T_0^{\varepsilon-\frac{1}{6}}}{1-e^{-2}} + \phi_1\frac{\log^{\frac{7}{2}}T}{T^{\frac{1}{4}}}+\phi_2\frac{\log^3{T}}{T^{3\varepsilon}},
\]
and
\begin{gather*}
    \phi_1\de 2\sqrt{2}\left(\frac{1}{6}-\varepsilon\right)^3\mathscr{S}_1\left(\sigma_0,T_0\right)\widehat{\lambda}\left(X\right)^2\left(1+\frac{\gamma}{\left(\frac{1}{6}-\varepsilon\right)\log{T_0}}+\frac{1}{2T_0^{\frac{1}{6}-\varepsilon}}\right), \\
    \phi_2\de 6\left(\frac{1}{6}-\varepsilon\right)^2\mathscr{S}_2\left(\sigma_0,T_0\right)\widehat{\lambda}\left(X\right)^2,
\end{gather*}
with $\mathscr{S}_1$ and $\mathscr{S}_2$ as in Theorem \ref{thm:SelbergMoment}, and $\widehat{\lambda}(X)$ from \eqref{eq:lambdaGB}.
\end{proposition}

\begin{proof}
Let $n$ and $m$ be positive integers not greater than $T/(2\pi)$. Define $z\de\left(\sigma,T,n,m\right)$ and $\hat{z}\de\left(\sigma,T,n/(n,m),m/(n,m)\right)$. Because $S(z)=S\left(\hat{z}\right)$, by Theorem \ref{thm:SelbergMoment} we have
\begin{flalign*}
\left|S\left(z\right)-\mathscr{S}\left(\hat{z}\right)\right| &\leq \frac{2\mathscr{S}_2\left(\sigma_0,T_0\right)nm}{(n,m)^2}T^{1-\sigma}\log{\frac{Tnm}{\pi(n,m)^2}} \\
&+ \mathscr{S}_1\left(\sigma_0,T_0\right)T^{1-\frac{\sigma}{2}}\left(\left(\frac{m}{n}\right)^{\sigma}\sqrt{\log{\frac{Tm}{\pi n}}}+\left(\frac{n}{m}\right)^{\sigma}\sqrt{\log{\frac{Tn}{\pi m}}}\right).
\end{flalign*}
Let $X\de T^{\alpha}$ for some $\alpha\in\left(0,1-\log{(2\pi)}/\log{T}\right)$. Then 
\begin{multline*}
\left|\mathop{\sum\sum}_{n,m\leq X} \frac{\lambda_X(n)\lambda_X(m)\left(S(z)-\mathscr{S}\left(\hat{z}\right)\right)}{(nm)^\sigma}\right| \leq \\ 
2\sqrt{2}\alpha^3\mathscr{S}_1\widehat{\lambda}\left(T^{\alpha}\right)^2\left(1+\frac{\gamma}{\alpha\log{T_0}}+\frac{1}{2T_0^\alpha}\right)T^{1-\frac{\sigma}{2}}\log^{\frac{7}{2}}{T}\\ +6\alpha^2\mathscr{S}_2\widehat{\lambda}\left(T^{\alpha}\right)^2T^{1+\alpha(1-2\sigma)+3\alpha-\sigma}\log^3{T}.
\end{multline*}
Now let $\sigma\geq 1/2+1/\log{X}$. By Lemmas \ref{lem:sellem1} and \ref{lem:sellem2} we have
\[
\left|\mathop{\sum\sum}_{n,m\leq X} \frac{\lambda_X(n)\lambda_X(m)\mathscr{S}\left(\hat{z}\right)}{(nm)^\sigma}\right| \leq T + \frac{1+\frac{2\sigma-1}{T_0^{\alpha}}}{1-e^{-2}}T^{1+\alpha(1-2\sigma)}
\]
since $\zeta(2-2\sigma)$ is negative for $\sigma\in(1/2,1)$. Take $\sigma_0\in(0,1)$ and define $\alpha\de 1/6-\varepsilon$ with $\varepsilon$ satisfying \eqref{eq:epsilon}. Then \eqref{eq:sigma} is a well-defined set and the bound for the integral now clearly follows.
\end{proof}

Observe that Proposition \ref{prop:Phi} implies 
\[
\int_{T}^{2T}\left|\Phi_X(\sigma+\ie t)\right|^2\dif{t} \leq T +O\left(T^{1-2\left(\frac{1}{6}-\varepsilon\right)\left(\sigma-\frac{1}{2}\right)}\right),
\]
uniformly for $\sigma$ on the set \eqref{eq:sigma} while $\sigma_0$ and $\varepsilon\in(0,1/6)$ are fixed. This is a slight generalisation of Selberg's result for $H=T$ since his bound follows for $\varepsilon=1/24$.

\begin{corollary}
\label{cor:PhiGeneral}
With assumptions and notations as in Proposition \ref{prop:Phi}, we have
\[
\frac{1}{2\pi}\int_{T}^{2T} \log{\left|\Phi_X(\sigma+\ie t)\right|}\dif{t} \leq \frac{1}{4\pi}\phi\left(\sigma_0,T_0,\varepsilon,T\right)T^{1-2\left(\frac{1}{6}-\varepsilon\right)\left(\sigma-\frac{1}{2}\right)}.
\]
\end{corollary}

\begin{proof}
The first inequality follows from Proposition \ref{prop:Phi}, and because $\log{(1+x)}\leq x$ and 
\[
\int_{a}^{b}\log{f(u)}\dif{u} \leq (b-a)\log{\left(\frac{1}{b-a}\int_{a}^{b}f(u)\dif{u}\right)}
\]
for $x\geq0$, and positive continuous functions $f(u)$ on $[a,b]\subset\R$. 
\end{proof}



\subsection{Explicit upper bound for $\int_{\sigma}^{\sigma_1}\left|\arg{\Phi_X(\tau+ \ie t)}\right|\dif{\tau}$.} Let $\sigma_1$ be as in Lemma \ref{lem:PhiArg} and let $w=\sigma_1+\left(2\sigma_1-1\right)e^{\ie\varphi}+\ie U$ where $\varphi\in[\pi/2,3\pi/2]$ and $U$ is not the ordinate of a zero of $\Phi(s)$. Assume that there is a function $\widehat{\Phi}\left(\sigma_1,\varphi,U,X\right)$ such that $\left|\Phi_X(w)\right|\leq \widehat{\Phi}\left(\sigma_1,\varphi,U,X\right)$. According to Proposition 4.10 in \cite{KLN}, for $\sigma\in\left(0,\sigma_1\right]$ we have
\begin{flalign}
\label{eq:argument}
\left|\arg{\Phi_X\left(\sigma+\ie U\right)}\right| &\leq \frac{1}{2\log{2}}\int_{\frac{\pi}{2}}^{\frac{3\pi}{2}}\log{\widehat{\Phi}\left(\sigma_1,\varphi,U\right)}\dif{\varphi} \nonumber \\ 
&+ \frac{\pi}{2\log{2}}\log{\frac{1+2^{-\sigma_1}h\left(\sigma_1\right)}{\left(1-2^{-\sigma_1}h\left(\sigma_1\right)\right)^2}}+\frac{\pi}{2}.
\end{flalign}
Trivially, 
\begin{equation}
\label{eq:phihat1}
\left|S_X(w)\right|\leq \widehat{\Phi}_1\left(\sigma_1,X\right)\de \frac{11}{5}X^{\sigma_1}\log{X}. 
\end{equation}
For the second part we will use the following convexity result.

\begin{lemma}
\label{lem:PL}
Let $s=\sigma+\ie t$ with $\sigma\in[1-\sigma_1,\sigma_1]$ where $\sigma_1\geq2.4$, and $s\neq1$. Then
\[
\left|\zeta(s)\right| \leq b_1\left(\sigma_1\right)^{\frac{\sigma_1-\sigma}{2\sigma_1-1}}\zeta\left(\sigma_1\right)^{\frac{\sigma+\sigma_1-1}{2\sigma_1-1}}\frac{\left|3\sigma_1-1+s\right|^{\frac{1}{2}+\sigma_1}}{|1-s|},
\]
where 
\[
b_1\left(\sigma_1\right)\de \sqrt{\frac{2}{\pi}}(2\pi)^{1-\sigma_1}e^{-\sigma_1+\frac{1}{12\sigma_1}+\frac{1}{90\sigma_1^3}}.
\]
\end{lemma}

\begin{proof}
By the functional equation for $\zeta(s)$ and the Stirling formula \eqref{eq:stieltjes} we have
\[
\left|\zeta\left(1-\sigma_1+\ie t\right)\right| \leq b_1\left(\sigma_1\right)\left|\sigma_1+\ie t\right|^{\sigma_1-\frac{1}{2}}.
\]
Applying the Phragm\'{e}n--Lindel\"{o}f theorem, see \cite[Theorem 2]{RademacherPL}, on the function $(1-s)\zeta(s)$ in the strip $\{s\in\C\colon 1-\sigma_1\leq\sigma\leq\sigma_1\}$ with $Q=2\sigma_1-1$, we obtain the main inequality.
\end{proof}

Lemma \ref{lem:PL} has the similar role as Lemma 3.1 in \cite{KLN}, except that in our case we need to consider a larger strip around the critical line. For similar results for $\sigma\in[1/2,1+\delta]$ while taking into account also sub-convexity bound for $\zeta\left(1/2+\ie t\right)$, see \cite[Lemma 2.7]{TrudgianImpr} and \cite[Corollary 2.2]{TrudgianImpr2}. 

Let $U\geq T_0>2\sigma_1-1$. Lemma \ref{lem:PL} implies 
\begin{equation}
\label{eq:phihat2}
\left|\zeta(w)\right|\leq \widehat{\Phi}_2\left(\sigma_1,\varphi,T_0,U\right)\de \frac{b_1\left(\sigma_1\right)^{-\cos{\varphi}}\zeta\left(\sigma_1\right)^{1+\cos{\varphi}}}{T_0-2\sigma_1+1}b_2\left(\sigma_1,T_0\right)^{\frac{1}{2}\left(\frac{1}{2}+\sigma_1\right)}U^{\frac{1}{2}+\sigma_1},
\end{equation}
where
\[
b_2\left(\sigma_1,T_0\right)\de \left(\frac{4\sigma_1-1}{T_0}\right)^2+\left(\frac{2\sigma_1-1}{T_0}+1\right)^2.
\]
We can now state the following.

\begin{proposition}
Let $U\geq T_0>2\sigma_1-1\geq 3.8$ and $U$ is not the ordinate of a zero of $\Phi_X(s)$. Then
\[
\left|\arg{\Phi_X\left(\sigma+\ie U\right)}\right| \leq \frac{\pi\left(\frac{1}{2}+\sigma_1\right)}{2\log{2}}\log{U} + \frac{\pi\sigma_1}{2\log{2}}\log{X} + \frac{\pi}{2\log{2}}\log{\log{X}} + b_3
\]
for $\sigma\in\left(0,\sigma_1\right]$, where
\begin{multline*}
b_3\left(\sigma_1,T_0\right)\de
\frac{\pi-2}{2\log{2}}\log{\zeta\left(\sigma_1\right)}+\frac{\log{b_1\left(\sigma_1\right)}}{\log{2}}+\frac{\pi\log{b_2\left(\sigma_1,T_0\right)}}{4\log{2}}\left(\frac{1}{2}+\sigma_1\right) \\
+\frac{\pi}{2\log{2}}\log{\frac{11\left(1+2^{-\sigma_1}h\left(\sigma_1\right)\right)}{5\left(T_0-2\sigma_1+1\right)\left(1-2^{-\sigma_1}h\left(\sigma_1\right)\right)^2}}+\frac{\pi}{2}
\end{multline*}
and $h\left(\sigma_1\right)$ is defined by \eqref{eq:h}. 
\end{proposition}

\begin{proof}
We can take $\widehat{\Phi}=\widehat{\Phi}_1\left(\sigma_1,X\right)\widehat{\Phi}_2\left(\sigma_1,\varphi,T_0,U\right)$, where $\widehat{\Phi}_1$ and $\widehat{\Phi}_2$ are defined by \eqref{eq:phihat1} and \eqref{eq:phihat2}, respectively. Now the result simply follows from \eqref{eq:argument}.
\end{proof}

\begin{corollary}
\label{cor:arg}
Let $T\geq T_0>2\sigma_1-1\geq 3.8$, $1>\sigma\geq1/2$, $X=T^{\frac{1}{8}}$, and $T$ or $2T$ is not the ordinate of a zero of $\Phi_X(s)$. Then
\begin{multline*}
\frac{1}{2\pi}\int_{\sigma}^{\sigma_1}\left|\arg{\Phi_X(\tau+ \ie 2T)}\right|+\left|\arg{\Phi_X(\tau+ \ie T)}\right|\dif{\tau} \\ 
\leq \frac{1}{2\pi}\left(\sigma_1-\frac{1}{2}\right)\left(\frac{\pi\left(9\sigma_1+4\right)}{8\log{2}}\log{T}+\frac{\pi}{2\log{2}}\log{\log{T}}-\frac{\pi\left(\frac{11}{2}-\sigma_1\right)}{2}+2b_3\right).
\end{multline*}
\end{corollary}



\subsection{Proof of Theorem \ref{thm:main}.} 
\label{sec:proof}
Firstly, we will provide some general bounds. Let $T\geq T_0\geq e^{24}\approx2.65\cdot10^{10}$ and $\sigma\in \left[1/2+8/\log{T},1/2+8/\log{T_0}\right]$. We can assume that $\Phi_X(s)$ does not have any zeros with imaginary parts equal to $T$ or $2T$ since the following inequalities can be extended by continuity principle also to these cases. Applying Corollaries~\ref{cor:PhiGeneral} and \ref{cor:arg} with $\varepsilon=1/24$, we obtain
\begin{equation}
\label{eq:proof1}
\int_{\sigma}^{1}N(\tau,2T)-N(\tau,T)\dif{\tau} \leq \alpha T^{1-\frac{1}{4}\left(\sigma-\frac{1}{2}\right)}+\beta\log{T}+\gamma\log{\log{T}}+\delta,
\end{equation}
where
\begin{gather*}
    \alpha\left(T_0\right)\de \frac{1}{4\pi}\phi\left(\frac{1}{2}+\frac{8}{\log{T_0}},T_0,\frac{1}{24},T_0\right), \\
    \beta\left(\sigma_1\right)\de \frac{\left(\sigma_1-\frac{1}{2}\right)\left(9\sigma_1+4\right)}{16\log{2}}, \quad \gamma\left(\sigma_1\right)\de \frac{\sigma_1-\frac{1}{2}}{4\log{2}}, \\
    \delta\left(\sigma_1,T_0\right)\de \left(\sigma_1-\frac{1}{2}\right)\left(\frac{2\sigma_1-11}{8}+\frac{b_3\left(\sigma_1,T_0\right)}{\pi}\right)+\frac{2^{-\sigma_1}h\left(\sigma_1\right)}{\left(1-2^{-\sigma_1}h\left(\sigma_1\right)\right)\pi\log{2}},
\end{gather*}
and $\sigma_1\geq 2.4$. By \eqref{eq:RvM} we have 
\[
N(2T)-N(T) \leq \frac{T}{2\pi}\log{T} + 0.22\log{T} + 0.6\log{\log{T}} + 5.
\]
Let $\sigma\in\left[1/2,1/2+8/\log{T}\right]$. Because $N(\tau,2T)-N(\tau,T)\leq \left(N(2T)-N(T)\right)/2$ and $T\leq e^2T^{1-\frac{1}{4}\left(\sigma-\frac{1}{2}\right)}$, we have
\begin{flalign}
\label{eq:proof2}
\int_{\sigma}^1 N(\tau,2T)-N(\tau,T)\dif{\tau} &\leq \frac{1}{2}\int_{\frac{1}{2}}^{\frac{1}{2}+\frac{8}{\log{T}}}N(2T)-N(T)\dif{\tau} \nonumber \\ 
&+\int_{\frac{1}{2}+\frac{8}{\log{T}}}^1 N(\tau,2T)-N(\tau,T)\dif{\tau} \nonumber \\
&\leq \left(\frac{2e^2}{\pi}+\alpha\right)T^{1-\frac{1}{4}\left(\sigma-\frac{1}{2}\right)}+\beta\log{T}+\gamma\log{\log{T}} \nonumber \\
&+\delta+0.88+2.4\frac{\log{\log{T_0}}}{\log{T_0}}+\frac{20}{\log{T_0}}.
\end{flalign}
Because the right-hand side of \eqref{eq:proof1} is always smaller than the right-hand side of \eqref{eq:proof2}, the latter inequality is true for all $\sigma\in\left[1/2,1/2+8/\log{T_0}\right]$.

With help of \eqref{eq:proof2} we are ready to estimate $N(\sigma,2T)-N(\sigma,T)$. Let $\sigma\in\left[1/2+4/\log{T},1/2+8/\log{T_0}\right]$. Then
\begin{flalign}
\label{eq:proof3}
N(\sigma,2T)-N(\sigma,T) &\leq \frac{\log{T}}{4}\int_{\sigma-\frac{4}{\log{T}}}^{1}N(\tau,2T)-N(\tau,T)\dif{\tau} \nonumber \\
&\leq \frac{e}{4}\left(\frac{2e^2}{\pi}+\alpha\right) T^{1-\frac{1}{4}\left(\sigma-\frac{1}{2}\right)}\log{T} \nonumber \\ 
&+ \frac{\beta}{4}\log^2{T}+\frac{\gamma}{4}\log{T}\cdot\log{\log{T}}+\left(\frac{\delta}{4}+0.51\right)\log{T}.
\end{flalign}
For $\sigma\in[1/2,1/2+4/\log{T}]$ we have
\[
N(\sigma,2T)-N(\sigma,T) \leq \frac{e}{4\pi} T^{1-\frac{1}{4}\left(\sigma-\frac{1}{2}\right)}\log{T}+\frac{1}{4}\log{T}. 
\]
The right-hand side of the latter inequality is obviously smaller than the right-hand side of \eqref{eq:proof3}, therefore this inequality is true for all $\sigma\in\left[1/2,1/2+8/\log{T_0}\right]$.

\begin{proof}[Proof of Theorem \ref{thm:main}]
Take $T_0=H_0$ and $\sigma_1=2.40764$. Since $\mathscr{S}_1\leq 219.618$ and $\mathscr{S}_2\leq611.578$ by Theorem \ref{thm:SelbergMoment}, we have $\alpha\left(H_0\right)<15291.986$, $\beta\left(\sigma_1\right)<4.416$, $\gamma\left(\sigma_1\right)<0.6881$ and $\delta\left(\sigma_1,H_0\right)<0$. Then the constants in Theorem \ref{thm:main} follows from \eqref{eq:proof3}.
\end{proof}

Observe that 
\[
\frac{e}{4}\left(\frac{2e^2}{\pi}+\alpha\left(T_0\right)\right) > \frac{e^3}{2\pi}\left(1+\frac{1}{8\left(e^2-1\right)}\right)>3.259,
\]
where the minimum is attained in the limit $T_0\to\infty$. This means that the leading term in Theorem \ref{thm:main} can be significantly improved if we take larger values for $T_0$, but its value could not be below $3.259$. For instance, if $T_0=10^{50}$, then $\alpha\left(T_0\right)<3.18$ where we also calculating new bounds for $\mathscr{S}_1$ and $\mathscr{S}_2$. Choosing $\sigma_1=2.4$, we get zero density estimate \eqref{eq:secondbound}.

\subsection*{Acknowledgements.} The author thanks Bryce Kerr for his comments on this manuscript, Daniele Dona, Harald Helfgott and Sebastian Zuniga Alterman for their interest in the explicit second power moment, Olivier Ramar\'{e} for lively discussions, and to Allysa Lumley for calculating the constants in \eqref{eq:KLN}. Finally, the author is grateful to his supervisor Tim Trudgian for continual guidance and support while writing this manuscript.


\begin{thebibliography}{DHZA19}

\bibitem[AdR11]{deReyna}
J.~Arias~de Reyna, {\em High precision computation of {R}iemann's zeta function
  by the {R}iemann-{S}iegel formula, {I}}, Math. Comp. {\bf 80} (2011),
  no.~274, 995--1009.

\bibitem[AM15]{Alirezaei}
G.~Alirezaei, R.~Mathar, {\em Analytical bounds on the average error probability for Nakagami fading channels}, 2015 Information Theory and Applications Workshop (ITA), San Diego, CA, 2015, pp.~54--63.

\bibitem[BBR12]{BBR}
D.~Berkane, O.~Bordell\`es, and O.~Ramar\'{e}, {\em Explicit upper bounds for
  the remainder term in the divisor problem}, Math. Comp. {\bf 81} (2012),
  no.~278, 1025--1051.

\bibitem[Che99]{Cheng}
Y.~Cheng, {\em An explicit upper bound for the {R}iemann zeta-function near the
  line {$\sigma=1$}}, Rocky Mountain J. Math. {\bf 29} (1999), no.~1, 115--140.

\bibitem[Con89]{ConreyAtLeast}
J.~B. Conrey, {\em At least two-fifths of the zeros of the {R}iemann zeta
  function are on the critical line}, Bull. Amer. Math. Soc. (N.S.) {\bf 20}
  (1989), no.~1, 79--81.

\bibitem[DHZA19]{DHA}
D.~Dona, H.~A. Helfgott, and S.~Zuniga~Alterman, {\em Explicit {$L^2$} bounds for the {R}iemann {$\zeta$} function}, preprint available at  arXiv:1906.01097v5.

\bibitem[Gab79]{Gabcke}
W.~Gabcke, {\em Neue {H}erleitung und explizite {R}estabsch\"{a}tzung der
  {R}iemann-{S}iegel {F}ormel}, Ph.D. thesis, Dissertation zur Erlangung des
  Doktorgrades der Mathematisch-Naturwissenschaftlichen Fakult\"{a}t der
  Georg-August-Universit\"{a}t zu G\"{o}ttingen, 1979.

\bibitem[GR15]{GradRyz}
I.~S. Gradshteyn and I.~M. Ryzhik, {\em Table of Integrals, Series, and
  Products}, 8th ed., Elsevier/Academic Press, Amsterdam, 2015.

\bibitem[Hia16]{Hiary}
G.~A. Hiary, {\em An alternative to {R}iemann-{S}iegel type formulas}, Math.
  Comp. {\bf 85} (2016), no.~298, 1017--1032.

\bibitem[HL21]{HLcritical}
G.~H. Hardy and J.~E. Littlewood, {\em The zeros of {R}iemann's zeta-function
  on the critical line}, Math. Z. {\bf 10} (1921), no.~3-4, 283--317.

\bibitem[HL23]{HLapprox1}
G.~H. Hardy and J.~E. Littlewood, {\em The {A}pproximate {F}unctional
  {E}quation in the {T}heory of the {Z}eta-{F}unction, with {A}pplications to
  the {D}ivisor-{P}roblems of {D}irichlet and {P}iltz}, Proc. London Math. Soc.
  (2) {\bf 21} (1923), 39--74.

\bibitem[HL29]{HLapprox2}
G.~H. Hardy and J.~E. Littlewood, {\em The {A}pproximate {F}unctional
  {E}quations for {$\zeta(s)$} and {$\zeta^2(s)$}}, Proc. London Math. Soc. (2)
  {\bf 29} (1929), no.~2, 81--97.

\bibitem[Ivi03]{Ivic}
A.~Ivi\'{c}, {\em The {R}iemann Zeta-Function}, Dover Publications, Inc.,
  Mineola, NY, 2003.

\bibitem[Jer19]{Jerby}
Y.~Jerby, {\em An approximate functional equation for the {R}iemann zeta
  function with exponentially decaying error}, preprint available at arXiv:1910.05754.

\bibitem[Jut83]{Jutila}
M.~Jutila, {\em Zeros of the zeta-function near the critical line}, Studies in
  pure mathematics, Birkh\"{a}user, Basel, 1983, pp.~385--394.

\bibitem[Kad13]{KadiriZeroDensity}
H.~Kadiri, {\em A zero density result for the {R}iemann zeta function}, Acta
  Arith. {\bf 160} (2013), no.~2, 185--200.

\bibitem[KK06]{KaratKor}
A.~A. Karatsuba and M.~A. Korol\"{e}v, {\em The behavior of the argument of the
  {R}iemann zeta function on the critical line}, Russian Math. Surveys {\bf 61}
  (2006), no.~3(369), 389--482.

\bibitem[KLN18]{KLN}
H.~Kadiri, A.~Lumley, and N.~Ng, {\em Explicit zero density for the {R}iemann
  zeta function}, J. Math. Anal. Appl. {\bf 465} (2018), no.~1, 22--46.

\bibitem[KV92]{KaratsubaVoronin}
A.~A. Karatsuba and S.~M. Voronin, {\em The {R}iemann Zeta-Function}, De
  Gruyter Expositions in Mathematics, vol.~5, Walter de Gruyter \& Co., Berlin,
  1992.

\bibitem[Leh56]{LehmerExt}
D.~H. Lehmer, {\em Extended computation of the {R}iemann zeta-function},
  Mathematika {\bf 3} (1956), 102--108.

\bibitem[Mat00]{MatsumotoRecent}
K.~Matsumoto, {\em Recent developments in the mean square theory of the
  {R}iemann zeta and other zeta-functions}, Number theory, Trends Math.,
  Birkh\"{a}user, Basel, 2000, pp.~241--286.

\bibitem[Mot83]{Motohashi}
Y.~Motohashi, {\em A note on the approximate functional equation for
  {$\zeta^{2}(s)$}}, Proc. Japan Acad. Ser. A Math. Sci. {\bf 59} (1983),
  no.~8, 393--396.

\bibitem[Olv74]{Olver}
F.~W.~J. Olver, {\em Asymptotics and Special Functions}, Academic Press, New
  York, 1974.

\bibitem[Pla17]{PlattRH}
D.~J. Platt, {\em Isolating some non-trivial zeros of zeta}, Math. Comp. {\bf
  86} (2017), no.~307, 2449--2467.

\bibitem[Pre84]{Preissmann}
E.~Preissmann, {\em Sur une in\'{e}galit\'{e} de {M}ontgomery-{V}aughan},
  Enseign. Math. (2) {\bf 30} (1984), no.~1-2, 95--113.

\bibitem[PT15]{PTZeta}
D.~J. Platt and T.~S. Trudgian, {\em An improved explicit bound on
  {$|\zeta(\frac 12+it)|$}}, J. Number Theory {\bf 147} (2015), 842--851.
  
\bibitem[PT19]{PTErrorTerm}
D.~J. Platt and T.~S. Trudgian, {\em The error term in the prime number theorem}, preprint available at arXiv:1809.03134.

\bibitem[Rad60]{RademacherPL}
H.~Rademacher, {\em On the {P}hragm\'{e}n-{L}indel\"{o}f theorem and some
  applications}, Math. Z. {\bf 72} (1959/1960), 192--204.

\bibitem[RS62]{RosserSchoenfeld}
J.~B. Rosser and L.~Schoenfeld, {\em Approximate formulas for some functions of
  prime numbers}, Illinois J. Math. {\bf 6} (1962), 64--94.

\bibitem[Sel46]{SelbergContrib}
A.~Selberg, {\em Contributions to the theory of the {R}iemann zeta-function},
  Arch. Math. Naturvid. {\bf 48} (1946), no.~5, 89--155.

\bibitem[Sie32]{Siegel}
C.~L. Siegel, {\em \"{U}ber {R}iemanns {N}achla{\ss} zur analytischen
  {Z}ahlentheorie}, Quellen und Studien zur Geschichte der Mathematik,
  Astronomie und Physik {\bf 2} (1932), 45--80.

\bibitem[Tit35]{TitchTheZeros}
E.~C. Titchmarsh, {\em The zeros of the {R}iemann zeta-function}, Proc. Roy.
  Soc. London {\bf 151} (1935), no.~873, 234--255.

\bibitem[Tit86]{Titchmarsh}
E.~C. Titchmarsh, {\em The Theory of the {R}iemann Zeta-Function}, 2nd ed., The
  Clarendon Press, Oxford University Press, New York, 1986.

\bibitem[Tru11]{TrudgianImpr}
T.~Trudgian, {\em Improvements to {T}uring's method}, Math. Comp. {\bf 80}
  (2011), no.~276, 2259--2279.

\bibitem[Tru14]{TrudgianUpper}
T.~S. Trudgian, {\em An improved upper bound for the argument of the {R}iemann
  zeta-function on the critical line {II}}, J. Number Theory {\bf 134} (2014),
  280--292.

\bibitem[Tru16]{TrudgianImpr2}
T.~Trudgian, {\em Improvements to {T}uring's method {II}}, Rocky Mountain J.
  Math. {\bf 46} (2016), no.~1, 325--332.

\bibitem[Tur43]{Turing}
A.~M. Turing, {\em A method for the calculation of the zeta-function}, Proc.
  London Math. Soc. (2) {\bf 48} (1943), 180--197.

\end{thebibliography}

\ifx\undefined\bysame
\newcommand{\bysame}{\leavevmode\hbox to3em{\hrulefill}\,}
\fi

\end{document}